\documentclass{dcdsOF} 
\usepackage{amsmath}
\usepackage{amssymb}
\usepackage{paralist}
\usepackage{epsfig} 
\usepackage{epstopdf} 
\usepackage[colorlinks=true]{hyperref}
\hypersetup{urlcolor=blue, citecolor=red}
\allowdisplaybreaks

\def\currentvolume{X}
 \def\currentissue{X}
  \def\currentyear{200X}
   \def\currentmonth{XX}
    \def\ppages{X--XX}
     \def\DOI{10.3934/dcds.2022067}

\newtheorem{theorem}{Theorem}[section]
\newtheorem{corollary}{Corollary}

\newtheorem{lemma}[theorem]{Lemma}
\newtheorem{proposition}{Proposition}

\newtheorem{problem}{Problem}
\newtheorem{example}{Example}
\theoremstyle{definition}
\newtheorem{definition}[theorem]{Definition}
\newtheorem{remark}{Remark}

\newcommand{\lip}{{\text{\rm Lip}}}

\newcommand{\norm}[1]{\left\Vert #1\right\Vert}
\newcommand{\nnorm}[1]{\lvert\!|\!| #1|\!|\!\rvert}

\title[Multiple ergodic averages for variable polynomials]
{Multiple ergodic averages for variable polynomials} 

\author[Andreas Koutsogiannis]{}


 \email{akoutsogiannis@math.auth.gr}

\begin{document}
\maketitle

\centerline{\scshape Andreas Koutsogiannis}
\medskip
{\footnotesize
 \centerline{Aristotle University of Thessaloniki}
   \centerline{Department of Mathematics}
   \centerline{Thessaloniki, 54124, Greece}
} 


\bigskip

\begin{center}\emph{Dedicated to the loving memory of Aris Deligiannis, a great mentor.}\end{center}

\bigskip

 \centerline{(Communicated by Zhiren Wang)}


\begin{abstract}
In this paper we study multiple ergodic averages for ``good'' variable polynomials. In particular, under an additional assumption, we show that these averages converge to the expected limit, making progress related to an open problem posted by Frantzikinakis (\cite[Problem~10]{F4}). These general convergence results imply several variable extensions of classical recurrence, combinatorial and number theoretical results which are presented as well.
\end{abstract}

%
%


\section{Introduction}

The study of multiple ergodic averages along polynomials dates back to 1977. Furstenberg, exploiting the $L^2$ limiting behavior (all the limits in this article are taken with respect to the $L^2$ norm, unless  otherwise stated),
as $N\to\infty,$ of
\begin{equation}\label{E:Furst}
\frac{1}{N}\sum_{n=1}^N T^{n}f_1\cdot T^{2n}f_2\cdot\ldots\cdot T^{\ell n}f_\ell,\end{equation}
where $\ell\in \mathbb{N},$ $(X,\mathcal{B},\mu,T)$ is a measure preserving system,\footnote{ I.e., $T:X\to X$ is an invertible measure preserving transformation on a standard Borel probability space $(X,\mathcal{B},\mu).$} and $f_1,\ldots,f_\ell\in L^\infty(\mu),$ provided (in \cite{Fu}) a purely ergodic theoretic proof of Szemer\'{e}di's theorem; every subset of natural numbers of positive upper density\footnote{For a set $ A \subseteq \mathbb{N} $ we define its \emph{upper density},
$ \bar{d}(A), $ as $ \bar{d}(A):=\limsup_{N}N^{-1}\cdot
|A\cap  \{1,\ldots, N \}|.$} contains arbitrarily long arithmetic progressions (a result that can be immediately obtained by combining Theorem~\ref{T:FMR} with Theorem~\ref{T:FCP} below).

A polynomial $p\in \mathbb{Q}[t]$ is an \emph{integer polynomial} if $p(\mathbb{Z})\subseteq \mathbb{Z}.$ It was Bergelson who first visualized the iterates $n, 2n,\ldots,\ell n$ in \eqref{E:Furst} as linear ``distinct enough'' integer polynomials. The integer polynomials $p_1,\ldots,p_\ell,$ $\ell\in \mathbb{N},$ are \emph{essentially distinct} if $p_i,$ $p_i-p_j$ are non-constant for all $i\neq j.$ Bergelson studied (initially in \cite{Be}), via the use of van der Corput's lemma, a crucial tool in ``reducing the complexity'' of the iterates,  averages of the form
\begin{equation}\label{E:Berg}
\frac{1}{N}\sum_{n=1}^N T^{p_1(n)}f_1\cdot\ldots\cdot T^{p_\ell(n)}f_\ell,\end{equation} for essentially distinct integer polynomials $p_1,\ldots,p_\ell$; this study eventually led to multidimensional polynomial extensions of Szemer\'{e}di's theorem (see \cite{BL}).

Bergelson and Leibman conjectured (in \cite{BL2}) that multiple ergodic averages of the form
\begin{equation}\label{E:Berg_Leib}
\frac{1}{N}\sum_{n=1}^N T_1^{p_1(n)}f_1\cdot\ldots\cdot T_\ell^{p_\ell(n)}f_\ell,\end{equation}
in any system, for multiple commuting $T_i$'s (i.e., $T_iT_j=T_jT_i$ for all $i,j$) and arbitrary integer polynomials $p_i$, always have limit (as $N\to\infty$). This conjecture was answered in the positive by Walsh, who actually showed it in greater generality (see \cite{W12}). 
 No specific expression of the limit was provided by the method.\footnote{ The conjecture corresponding to that of
Bergelson and Leibman about iterates which are integer parts of real polynomials, was shown in \cite{K1}.}

One of the questions that someone is called upon to answer is under which conditions, either on the polynomials or the system, we can explicitly find the limit of the aforementioned expressions. In particular, whether we can find families of polynomials for which we have convergence in a general system to a specific expression (see more about the ``expected'' limit below); then we can get a number of interesting applications, e.g., find the corresponding arithmetic configurations on ``large'' subsets of integers. For instance, 
showing that the characteristic factor coincides with the nilfactor of the system, and exploiting the equidistribution property of the corresponding polynomial sequence in nilmanifolds (all these notions will be defined
later), Frantzikinakis proved (in \cite{F2}) that the expression
\begin{equation}\label{E:Frang}
\frac{1}{N}\sum_{n=1}^N T^{[p(n)]}f_1\cdot T^{2[p(n)]}f_2\cdot\ldots\cdot T^{\ell[p(n)]}f_\ell,\end{equation} where $p\in \mathbb{R}[t]$ with $p(t)\neq cq(t)+d,$ $c,d\in \mathbb{R},$ $q\in \mathbb{Q}[t],$  
has the same limit (as $N\to\infty$), in \emph{any} system, as \eqref{E:Furst}; obtaining a refinement of Szemer\'{e}di's theorem.\footnote{ Such polynomials $p$ have the property that for every $\lambda\in\mathbb{R}\setminus\{0\},$ $\lambda p$ has at least one non-constant irrational coefficient, which is exactly the case (via Weyl's criterion) when the corresponding sequence $(\lambda p(n))_n$ is equidistributed.}

Generalizing the condition to multiple polynomials, following Frantzikinakis' approach, Karageorgos and the author showed (in \cite{KK}) that for strongly independent real polynomials $p_1,\ldots,p_\ell$ (i.e., any non-trivial linear combination of the $p_i$'s with scalars from $\mathbb{R}$ has at least one non-constant irrational coefficient) the expression
\begin{equation}\label{E:Kar_Kouts}
\frac{1}{N}\sum_{n=1}^N T^{[p_1(n)]}f_1\cdot\ldots\cdot T^{[p_\ell(n)]}f_\ell,\end{equation}
has the ``expected'' limit. In order to explain what is meant by ``expected'' limit, we need to recall the ergodicity and weakly mixing notions. $T$ is \emph{ergodic} if $T^{-1}A=A$ implies $\mu(A)\in \{0,1\}$; $T$ is \emph{weakly mixing} if $T\times T$ is ergodic. Here, by ``expected'' limit we mean, in case $T$ is ergodic, that the limit is equal to $\prod_{i=1}^\ell \int f_i\;d\mu,$ whereas, in the general case, it is equal to $\prod_{i=1}^\ell \mathbb{E}(f_i|\mathcal{I}(T)),$ where  $\mathbb{E}(f_i|\mathcal{I}(T))$ is the conditional expectation of $f_i$ with respect to the $\sigma$-algebra of the $T$-invariant sets (notice here the connection to independence in probability). Furstenberg showed in \cite{Fu} that for a weakly mixing $T,$ \eqref{E:Furst} converges to the expected limit; under the same assumption on $T$, Bergelson showed in \cite{Be} that \eqref{E:Berg} converges to the same limit as well.  

We extend the distinctness property of the sequences of iterates of \eqref{E:Kar_Kouts} to sequences of real variable polynomials:

\begin{definition}[\cite{F4}]\label{D:Good}
The sequence $(p_N)_N,$ where $p_N\in \mathbb{R}[t], N\in \mathbb{N},$ is \emph{good} if the polynomials
have bounded degree and for every non-zero $\alpha\in \mathbb{R}$ we have
\begin{equation}\label{E:good}
\lim_{N\to\infty} \frac{1}{N}\sum_{n=1}^N e^{ip_N(n)\alpha}=0.
\end{equation}
The sequence of $\ell$-tuples of variable polynomials $(p_{1,N}, \ldots, p_{\ell,N})_N,$ where $p_{i,N} \in \mathbb{R}[t],$ $N \in\mathbb{N},$
$1\leq i \leq \ell,$ is \emph{good} if every non-trivial linear combination of the sequences $(p_{1,N})_N, \ldots, (p_{\ell,N})_N$ is good.
\end{definition}



\begin{example}[\cite{F4}]\label{Ex:1} {\rm
 For $\ell=2,$ the pair $(p_{1,N}, p_{2,N})_N,$ where $p_{1,N}(n)=n/N^a,$ $p_{2,N}(n)=n/N^b,$ $N,n\in \mathbb{N},$ $0<a<b<1,$ is good.}
 \end{example}

\begin{example}[\cite{F4}]\label{Ex:2} {\rm For $\ell\in \mathbb{N},$ the $\ell$-tuple $(p_{1,N},\ldots,p_{\ell,N})_N,$ where $p_{i,N}(n)=n^i/N^a,$ $1\leq i\leq \ell,$ $N,n\in \mathbb{N},$ $0<a<1,$ is good.}
\end{example}

For the class of good polynomial sequences, Frantzikinakis stated the following problem:

\begin{problem}[Problem 10, \cite{F4}]\label{P:10}
Let $(p_{1,N}, \ldots, p_{\ell,N})_N$ be a good $\ell$-tuple of variable polynomials. Is it true that, for every ergodic system $(X,\mathcal{B},\mu,T)$ and functions $f_1,\ldots,f_\ell\in L^\infty(\mu),$ we have
\[\lim_{N\to\infty}\frac{1}{N}\sum_{n=1}^N T^{[p_{1,N}(n)]}f_1\cdot\ldots\cdot T^{[p_{\ell,N}(n)]}f_\ell= \int f_1\;d\mu\cdot\ldots\cdot\int f_\ell\;d\mu  \;?\]
\end{problem}


As stated in \cite{F4}, Problem~\ref{P:10} is interesting even in the special cases of Examples~\ref{Ex:1} and ~\ref{Ex:2}.


Showing that \eqref{E:Frang} has the same limit as \eqref{E:Furst} for $p\in \mathbb{R}[t]$ with $p(t)\neq cq(t)+d,$ $c,d\in \mathbb{R},$ $q\in \mathbb{Q}[t],$ which follows from \cite[Theorem~2.2]{F2}, one comes to the following problem, which is a natural generalization of Frantzikinakis' result to good-variable-polynomials:

\begin{problem}\label{P:2}
Let $(p_N)_N$ be a good polynomial sequence. Is it true that, for every $\ell\in \mathbb{N},$ system $(X,\mathcal{B},\mu,T),$ and $f_1,
\ldots,f_\ell\in L^\infty(\mu),$ we have
\begin{equation*}
 \begin{split}
&\lim_{N\to\infty}\frac{1}{N}\sum_{n=1}^N T^{[p_N(n)]}f_1\cdot T^{2[p_N(n)]}f_2\cdot\ldots\cdot T^{\ell[p_N(n)]}f_\ell \\
=&\lim_{N\to\infty}\frac{1}{N}\sum_{n=1}^N T^{n}f_1\cdot T^{2n}f_2\cdot\ldots\cdot T^{\ell n}f_\ell \;?
\end{split}
\end{equation*}
\end{problem}

As mentioned in \cite{F4}, the $\ell=1$ case of Problem~\ref{P:10} (which also coincides with the $\ell=1$ case of Problem~\ref{P:2})
can be easily obtained by using the spectral theorem. For general $\ell\in \mathbb{N},$ we make progress towards the solution of both Problems~\ref{P:10} and \ref{P:2}. In particular, under some additional assumptions on the coefficients of the good variable polynomials, we show two general results: Theorems~\ref{T:main} and ~\ref{T:mc}. In this introductory section, we will present an easier application of each of them, which still covers both Examples~\ref{Ex:1} and ~\ref{Ex:2}.




To this end, we first recall the set of sublinear logarithmico-exponential Hardy field functions (of polynomial degree $0$) which converge (as $x\to\infty$) to ($\pm$) infinity:\footnote{ Let $R$ be the collection of equivalence classes of real valued functions defined on some halfline $(c,\infty),$ $c\geq 0,$ where two functions that agree eventually are identified. These classes are called \emph{germs} of functions.  A \emph{Hardy field} is a subfield of the ring $(R, +, \cdot)$ that is closed under differentiation. Here, we use the word \emph{function} when we refer to elements of $R$ (understanding that all the operations defined and statements made for elements of $R$ are considered only for sufficiently large $x\in \mathbb{R}$). We say that $g$ is a \emph{logarithmico-exponential Hardy field function}, and we write $g\in \mathcal{LE},$ if it belongs to a Hardy field of real valued functions and it is defined on some $(c,+\infty),$ $c\geq 0,$ by a finite combination of symbols $+, -, \times, \div, \sqrt[n]{\cdot}, \exp, \log$ acting on the real variable $x$ and on real constants. For more on Hardy field functions, see \cite{F3, F2, Hard}.}
\[\mathcal{SLE}:=\{g\in \mathcal{LE}:\;1\prec g(x)\prec x\}\]
(we write $g_2\prec g_1\text{ if }|g_1(x)|/|g_2(x)|\to \infty\text{ as }x\to\infty$).
Next, we define an appropriate set of coefficients: For $g_i\in \mathcal{SLE},$ $1\leq i\leq l,$ $l\in\mathbb{N},$ with $g_l\prec \ldots \prec g_1,$\footnote{ The different growth relation between the $g_i$'s, $1\leq i\leq l,$ is postulated to avoid cases as, e.g., $(N+1)^{-1/2}-N^{-1/2}\sim N^{-3/2},$ since $g(x)=x^{3/2}$ is not sublinear (here we write $g_2\sim g_1$ if $g_1(x)/g_2(x)\to c\in \mathbb{R}\setminus\{0\}$).} let  $\mathcal{C}(g_1,\ldots,g_l)$ be the set of all linear combinations of reciprocals of the $g_i$'s, i.e.,
\[\mathcal{C}(g_1,\ldots,g_l):=\left\{\sum_{i=1}^{l}\frac{\rho_i}{g_i}:\;\rho_i\in \mathbb{R}\right\}.\]

Extending the definition from \cite{KK}, we say that the sequence of $\ell$-tuple of variable polynomials $(p_{1,N},\ldots,p_{\ell,N})_N,$ where for each $1\leq i\leq \ell,$ $p_{i,N}$ has the form:
\begin{equation}\label{E:polyv}
	\begin{split}
	&\quad p_{i,N}(n)=a_{i,d_i,N}n^{d_i}+\dots+a_{i,1,N}n+a_{i,0,N},
	\\
	& \text{with}\;\;(a_{i,0,N})_N\;\;\text{bounded, and}\;\;a_{i,j,\cdot}\in \mathcal{C}(g_1,\ldots,g_l),\;1\leq j\leq d_i,
	\end{split}
    \end{equation}
is \emph{strongly independent} if
for any $(\lambda_1,\ldots,\lambda_\ell)\in \mathbb{R}^{\ell}\setminus\{\vec{0}\}$ we have that
$\sum_{i=1}^\ell \lambda_i p_{i,N}(n)$ is a non-constant polynomial in $n$. 
For example, the following triple of variable polynomials is strongly independent:
\[\left(\left(\frac{\sqrt{2}}{g_3(N)}+\frac{1}{g_4(N)}\right)n^3-\frac{31}{g_5(N)}n+1,\frac{1}{g_3(N)}n^3,\left(\frac{\sqrt{3}}{g_1(N)}-\frac{17}{g_2(N)}\right)n^2\right)_N,\]
where $g_5(x):=\log x\prec g_4(x):=\log x\cdot \log\log x\prec g_3(x):=x^{1/2}\prec g_2(x):=x^{\pi/4}\log^{3/2}x\prec g_1(x):=x^{e/3}.$\footnote{ If $p_{1,N}(n):=\Big(\frac{\sqrt{2}}{g_3(N)}+\frac{1}{g_4(N)}\Big)n^3-\frac{31}{g_5(N)}n+1,$ $p_{2,N}(n):=\frac{1}{g_3(N)}n^3,$ and $p_{3,N}(n):=\Big(\frac{\sqrt{3}}{g_1(N)}-\frac{17}{g_2(N)}\Big)n^2,$ we have that $\lambda_1p_{1,N}+\lambda_2p_{2,N}+\lambda_3p_{3,N}$ is constant only when $\lambda_1=\lambda_2=\lambda_3=0.$}
Regarding Problem~\ref{P:10}, i.e., multiple variable polynomial sequences, we have the following result:

\begin{theorem}\label{T:main_new}
For $\ell\in \mathbb{N},$ let $(p_{1,N}, \ldots, p_{\ell,N})_N$ be a strongly independent $\ell$-tuple of polynomials as in \eqref{E:polyv}. Then, for every ergodic system $(X,\mathcal{B},\mu,T)$ and $f_1,\ldots,f_\ell\in L^\infty(\mu),$ we have
\begin{equation*}\lim_{N\to\infty}\frac{1}{N}\sum_{n=1}^N T^{[p_{1,N}(n)]}f_1\cdot\ldots\cdot T^{[p_{\ell,N}(n)]}f_\ell=\int f_1\;d\mu\cdot\ldots\cdot\int f_\ell\;d\mu.\end{equation*}
\end{theorem}



Regarding
Problem 2, our result is the following theorem:

\begin{theorem}\label{T:mc_new}
Let $(p_N)_N$ be a polynomial sequence as in \eqref{E:polyv} with $p_N(n)$ non-constant in $n$. Then, for every $\ell\in \mathbb{N},$ system $(X,\mathcal{B},\mu,T),$ and $f_1,
\ldots,f_\ell\in L^\infty(\mu),$ we have
\begin{equation*}
 \begin{split}
&\lim_{N\to\infty}\frac{1}{N}\sum_{n=1}^N T^{[p_N(n)]}f_1\cdot T^{2[p_N(n)]}f_2\cdot\ldots\cdot T^{\ell[p_N(n)]}f_\ell\\
=&\lim_{N\to\infty}\frac{1}{N}\sum_{n=1}^N T^{n}f_1\cdot T^{2n}f_2\cdot\ldots\cdot T^{\ell n}f_\ell.
\end{split}
\end{equation*}
\end{theorem}

Very few convergence results for averages with polynomial iterates, in which we can explicitly find the limit, exist. Results for variable polynomials are even scarcer. We will conclude this introduction, by mentioning some of them. Kifer (\cite{K}) studied multiple averages for variable polynomials of the form $p_{i,N}(n)=p_i(n)+q_i(N),$ with $p_i$'s essentially distinct and $q_i(\mathbb{Z})\subseteq \mathbb{Z}$ for a weakly mixing transformation $T.$ Similarly, for more general polynomials, Kifer studied averages for strongly mixing ``enough'' transformations. Finally, Frantzikinakis (in \cite{F2}) found characteristic factors (see Definition~\ref{D:CF}) for averages with variable polynomial iterates $p_{i,N},$ with leading coefficients independent of $N$. It is the arguments from this article (\cite{F2}) that we will adapt, in order to find characteristic factors for the averages appearing in Theorems ~\ref{T:main_new} and ~\ref{T:mc_new} as well, which is one of the main two ingredients of the proof (the second one is the equidistribution of particular sequences for which we adapt arguments from \cite{F1}).

\subsection*{Notation} We denote by $\mathbb{N}=\{1,2,\ldots\},$ $\mathbb{Z},$ $\mathbb{Q},$ $\mathbb{R},$ and $\mathbb{C}$ the sets of natural, integer, rational, real and complex numbers respectively. For a function $f:X\to\mathbb{C}$ on a space $X$
with a transformation $T:X\to X,$ we denote by $Tf$ the composition $f\circ T.$  For $s\in \mathbb{N},$ $\mathbb{T}^s =\mathbb{R}^s/\mathbb{Z}^s $ denotes the $s$ dimensional torus. For  $a,b\geq 0$ we write $a\ll b,$ if there exists $C>0$ such that $a\leq C b.$

\section{Main results and applications}\label{S:applications}

Here we will state our most general results and some applications. For the proofs, we follow \cite{F3} and \cite{KK}, adapting the corresponding arguments to the variable polynomial case.


We first cover Problem~\ref{P:10} for a subclass of good polynomial sequences:

\begin{theorem}\label{T:main}
For $\ell\in \mathbb{N},$ let $(p_{1,N}, \ldots, p_{\ell,N})_N$ be a good and super nice\footnote{ The ``super niceness'' property is rather technical
and will be defined in Section~\ref{S:4}.} $\ell$-tuple of polynomials. Then, for every ergodic system $(X,\mathcal{B},\mu,T)$ and $f_1,\ldots,f_\ell\in L^\infty(\mu),$ we have
\begin{equation}\label{E:main33}\lim_{N\to\infty}\frac{1}{N}\sum_{n=1}^N T^{[p_{1,N}(n)]}f_1\cdot\ldots\cdot T^{[p_{\ell,N}(n)]}f_\ell=\int f_1\;d\mu\cdot\ldots\cdot\int f_\ell\;d\mu.\end{equation}
\end{theorem}

We also cover the following case of Problem~\ref{P:2}:

\begin{theorem}\label{T:mc}
Let $(p_N)_N\subseteq \mathbb{R}[t]$ be a good polynomial sequence such that, for all $\ell\in \mathbb{N},$ $(p_{N},2p_N,$ $\ldots,\ell p_N)_N$ is super nice. Then, for every $\ell\in \mathbb{N},$ system $(X,\mathcal{B},\mu,T),$ and $f_1,
\ldots,f_\ell\in L^\infty(\mu),$ we have
\begin{equation}\label{E:mc3}
 \begin{split}
&\lim_{N\to\infty}\frac{1}{N}\sum_{n=1}^N T^{[p_N(n)]}f_1\cdot T^{2[p_N(n)]}f_2\cdot\ldots\cdot T^{\ell[p_N(n)]}f_\ell\\
=&\lim_{N\to\infty}\frac{1}{N}\sum_{n=1}^N T^{n}f_1\cdot T^{2n}f_2\cdot\ldots\cdot T^{\ell n}f_\ell.
\end{split}
\end{equation}
\end{theorem}

We will show that Theorem~\ref{T:main} implies Theorem~\ref{T:main_new} (resp. Theorem~\ref{T:mc} implies Theorem~\ref{T:mc_new}), and that it holds for any polynomial family $\{p_1,\ldots,p_\ell\}$ (resp. $\{p,2p,\ldots,\ell p\}$ for Theorem~\ref{T:mc}) which is independent of $N$ and for which non-trivial linear combinations of its members satisfy \eqref{E:good}.
In particular, it generalizes \cite[Theorem~2.1]{KK} for strongly independent polynomials (the same is true for Theorem~\ref{T:mc} for the single polynomial case).

The approach we follow to show these results is similar to the one in \cite{F2,KK}, with a few extra twists. 
Namely, one has to find the characteristic factors of \eqref{E:main33} and \eqref{E:mc3}, 
and show some equidistribution results in nilmanifolds. 
The ``super niceness'' property (Definition~\ref{D:super_nice}) will be introduced so we can deal with the former, while the ``goodness'' property (Definition~\ref{D:Good}) implies the latter.

As was mentioned in the previous section, the ergodicity assumption in Theorem~\ref{T:main} can be dropped.\footnote{ The limit in this case is equal to $\prod_{i=1}^\ell \mathbb{E}(f_i|\mathcal{I}(T)).$ Indeed, if $\mu=\int \mu_t\;d\lambda(t)$ denotes the ergodic decomposition of $\mu,$ it suffices to show that if $\mathbb{E}(f_i\vert \mathcal{I}(T))=0$ for some $i,$ then the averages converge to $0.$ Since $\mathbb{E}(f_i\vert \mathcal{I}(T))=0,$ we have that $\int f_i \; d\mu_t=0$ for $\lambda$-a.e. $t.$ By \eqref{E:main33}, we have that the averages go to $0$ in $L^2(\mu_t)$ for $\lambda$-a.e. $t,$ hence the limit is equal to $0$ in $L^2(\mu)$ by the Dominated Convergence Theorem.} Hence, the theorems hold for any system; their strong nature is also reflected in the fact that they have immediate recurrence 
and combinatorial implications which we discuss next.






\subsection{Single sequence consequences}\label{Ss:2.1} We first deal with a single variable polynomial sequence, assuming the validity of Theorem~\ref{T:mc}.

\subsubsection{Recurrence} The following theorem due to Furstenberg will help us obtain recurrence results:

\begin{theorem}[Furstenberg Multiple Recurrence Theorem, ~\cite{Fu}]\label{T:FMR}
Let $(X,\mathcal{B},\mu,T)$ be a system. Then, for every $\ell\in \mathbb{N}$ and every set $A\in \mathcal{B}$ with $\mu(A)>0,$ we have
\[\liminf_{N\to\infty}\frac{1}{N}\sum_{n=1}^N \mu(A\cap T^{-n}A\cap T^{-2n}A\cap\ldots\cap T^{-\ell n}A)>0.
\]
\end{theorem}

\begin{remark}
As we mentioned before, the liminf in the expression of Theorem~\ref{T:FMR} is actually a limit.
\end{remark}

Combining Theorem~\ref{T:mc} and Theorem~\ref{T:FMR} (with the choices $f_i:=1_A$),
we get the following corollary:

\begin{corollary}\label{T:mcr}
Let $(p_N)_N\subseteq\mathbb{R}[t]$ be as in Theorem~\ref{T:mc}.
Then, for every $\ell\in \mathbb{N},$ every system $(X,\mathcal{B},\mu,T),$ and every set $A\in \mathcal{B}$ with $\mu(A)>0,$ we have
\[\lim_{N\to\infty}\frac{1}{N}\sum_{n=1}^N \mu\left(A\cap T^{-[p_N(n)]}A\cap T^{-2[p_N(n)]}A\cap\ldots\cap T^{-\ell[p_N(n)]}A\right)>0.\]
\end{corollary}


\subsubsection{Combinatorics} Via Furstenberg's Correspondence Principle, one gets combinatorial results from recurrence ones. We present here a reformulation of this principle from \cite{B0}.

\begin{theorem}[Furstenberg Correspondence Principle,~\cite{Fu},~\cite{B0}]\label{T:FCP}
Let $E$ be a subset of integers. There exists a system $(X,\mathcal{B},\mu,T)$ and a set $A\in \mathcal{B}$ with $\mu(A)=\bar{d}(E)$  such that \begin{equation}\label{E:FCP}
\bar{d}(E\cap(E-n_1)\cap\ldots\cap (E-n_\ell))\geq \mu(A\cap T^{-n_1}A\cap\ldots\cap T^{-n_\ell}A)
\end{equation}
for every $\ell\in\mathbb{N}$ and $n_1,\ldots,n_\ell\in \mathbb{Z}.$
\end{theorem}

Using Corollary~\ref{T:mcr} and Theorem~\ref{T:FCP}, we have the following combinatorial result:

\begin{corollary}\label{T:s1}
Let $(p_N)_N\subseteq\mathbb{R}[t]$ be as in Theorem~\ref{T:mc}. Then, for every $\ell\in \mathbb{N}$ and every set $E\subseteq \mathbb{N}$ with $\bar{d}(E)>0,$ we have
\[\liminf_{N\to\infty}\frac{1}{N}\sum_{n=1}^N \bar{d}\left(E\cap(E-[p_N(n)])\cap(E-2[p_N(n)])\cap\ldots\cap (E-\ell[p_N(n)])\right)>0.\]
\end{corollary}


Hence, we immediately get the following refinement of Szemer{\'e}di's theorem:

\begin{corollary}\label{C:Sz}
Let $(p_N)_N\subseteq\mathbb{R}[t]$ be as in Theorem~\ref{T:mc}. Then, for every $\ell\in \mathbb{N}$, every set $E\subseteq \mathbb{N}$ with $\bar{d}(E)>0$ contains arithmetic progressions of the form: \[\{m, m+[p_N(n)],m+2[p_N(n)],\ldots,m+\ell[p_N(n)]\},\]
for some $m\in \mathbb{Z},$ $N\in \mathbb{N},$ and $1\leq n\leq N,$ with $[p_N(n)]\neq 0.$
\end{corollary}

\subsection{Multiple sequences consequences}\label{Ss:2.2}
As in Subsection~\ref{Ss:2.1}, assuming the validity of Theorem~\ref{T:main}, we have various implications for multiple variable polynomial sequences.

\subsubsection{Recurrence} Our first recurrence result is the following (we skip the proof as the argument is the same as the one in \cite[Theorem~2.8]{F1}):

\begin{theorem}\label{T:r1}
For $\ell\in \mathbb{N},$ let $(p_{1,N}, \ldots, p_{\ell,N})_N$ be as in Theorem~\ref{T:main}. If $(X,\mathcal{B},\mu,T)$ is a system and $A_0, A_1,\ldots, A_\ell\in \mathcal{B}$ such that
\[\mu\left(A_0\cap T^{k_1}A_1\cap\ldots\cap T^{k_\ell}A_\ell\right)=\alpha>0\]
for some $k_1,\ldots,k_\ell\in \mathbb{Z},$ then
\begin{equation*}
\lim_{N\to\infty}\frac{1}{N}\sum_{n=1}^N \mu\left(A_0\cap T^{-[p_{1,N}(n)]}A_1\cap\ldots\cap T^{-[p_{\ell,N}(n)]}A_\ell\right)\geq \alpha^{\ell+1}.
\end{equation*}
\end{theorem}

Setting $A_i=A$ and $k_i=0$ we immediately get the following:

\begin{corollary}\label{C:r1}
For $\ell\in \mathbb{N},$ let $(p_{1,N}, \ldots, p_{\ell,N})_N$ be as in Theorem~\ref{T:main}. Then, for every system $(X,\mathcal{B},\mu,T)$ and every set $A\in \mathcal{B},$ we have \begin{equation*}\label{E:r1_primes_2}
\lim_{N\to\infty}\frac{1}{N}\sum_{n=1}^N \mu\left(A\cap T^{-[p_{1,N}(n)]}A\cap\ldots\cap T^{-[p_{\ell,N}(n)]}A\right)\geq (\mu(A))^{\ell+1}.
\end{equation*}
\end{corollary}



\subsubsection{Combinatorics}

Theorem~\ref{T:r1}, via \cite[Proposition~3.3]{F5}, which is a variant of Theorem~\ref{T:FCP} for several sets, implies the following (we are skipping the routine details):

\begin{theorem}\label{T:r2}
For $\ell\in \mathbb{N},$ let $(p_{1,N}, \ldots, p_{\ell,N})_N$ be as in Theorem~\ref{T:main}. If $E_0, E_1,\ldots,$ $E_\ell\subseteq \mathbb{N}$ are such that
\[\bar{d}(E_0\cap(E_1+k_1)\cap\ldots\cap(E_\ell+k_\ell))=\alpha>0\]
for some $k_1,\ldots,k_\ell\in \mathbb{Z},$ then
\[\liminf_{N\to\infty}\frac{1}{N}\sum_{n=1}^N \bar{d}(E_0\cap (E_1-[p_{1,N}(n)])\cap\ldots\cap (E_\ell-[p_{\ell,N}(n)]))\geq \alpha^{\ell+1}.\]
\end{theorem}

Setting $E_i=E$ and $k_i=0$ in the previous result, we get:

\begin{corollary}\label{C:r2}
For $\ell\in \mathbb{N},$ let $(p_{1,N}, \ldots, p_{\ell,N})_N$ be as in Theorem~\ref{T:main}. Then, for every set $E\subseteq \mathbb{N},$ we have \[\liminf_{N\to\infty}\frac{1}{N}\sum_{n=1}^N \bar{d}(E\cap (E-[p_{1,N}(n)])\cap\ldots\cap (E-[p_{\ell,N}(n)]))\geq (\bar{d}(E))^{\ell+1}.\]
\end{corollary}

So, we immediately obtain the following combinatorial result:

\begin{corollary}\label{C:r22}
For $\ell\in \mathbb{N},$ let $(p_{1,N}, \ldots, p_{\ell,N})_N$ be as in Theorem~\ref{T:main}. Then every set $E\subseteq \mathbb{N}$ with $\bar{d}(E)>0$ contains arithmetic configurations of the form
\[\{m, m+[p_{1,N}(n)],m+[p_{2,N}(n)],\ldots,m+[p_{\ell,N}(n)]\},\]
for some $m\in \mathbb{Z},$ $N\in\mathbb{N},$ and $1\leq n\leq N,$ with $[p_{i,N}(n)]\neq 0,$ for all $1\leq i\leq \ell.$
\end{corollary}

A set $E\subseteq \mathbb{N}$ is called \emph{syndetic} if finitely many translations of it cover $\mathbb{N}.$ The cardinality of such a set of translations is a \emph{syndeticity constant} of $E$.
Applying Theorem~\ref{T:r2} to syndetic sets $E_0, E_1,\ldots, E_\ell$ and $\alpha=(\prod_{i=1}^\ell r_i)^{-1},$ where $r_i$ is a syndeticity constant of $E_i,$ $1\leq i\leq \ell,$ we have:

\begin{corollary}\label{C:r23}
For $\ell\in \mathbb{N},$ let $(p_{1,N}, \ldots, p_{\ell,N})_N$ be as in Theorem~\ref{T:main}. If $E_0, E_1,\ldots,$ $E_\ell\subseteq \mathbb{N}$ are syndetic sets, then there exist $m\in \mathbb{Z},$ $N\in\mathbb{N},$ and $1\leq n\leq N,$ with $[p_{i,N}(n)]\neq 0,$ for all $1\leq i\leq \ell,$ such that
\[m\in E_0, m+[p_{1,N}(n)]\in E_1,\ldots, m+[p_{\ell,N}(n)]\in E_\ell.\]
\end{corollary}

In particular, for a syndetic set $E\subseteq \mathbb{N}$, setting $E_i=c_i E:=\{c_in:\;n\in E\},$ $0\leq i\leq \ell,$ where  $c_0,c_1,\ldots,c_\ell\in \mathbb{N},$ Corollary~\ref{C:r23} above implies that we can find $x_0,x_1,\ldots,x_\ell\in E,$ $N\in \mathbb{N},$ and $1\leq n\leq N,$ that solve the system of equations $c_ix_i-c_{i-1}x_{i-1}  =  [p_{i,N}(n)],$ $1\leq i\leq \ell.$

\subsubsection{Topological dynamics} Let $(X, T)$ be a (topological) dynamical
system, i.e., $(X, d)$ is a compact metric space and $T : X \to X$ an invertible continuous transformation. $T$ (and consequently the system) is \emph{minimal}, if, for all $x\in X,$ we have $\overline{\{T^n x:\;n\in\mathbb{N}\}}=X.$

Analogously to \cite[Theorem~2.5]{KK}, we get the following result:

\begin{theorem}\label{T:topol}
For $\ell\in \mathbb{N},$ let $(p_{1,N}, \ldots, p_{\ell,N})_N$ be as in Theorem~\ref{T:main}. If $(X,T)$ is a minimal dynamical system, then, for a residual and $T$-invariant set of $x\in X,$ we have
\begin{equation}\label{E:topol}
\overline{\Big\{(T^{[p_{1,N}(n)]}x,\ldots,T^{[p_{\ell,N}(n)]}x):\;N\in\mathbb{N},\;1\leq n\leq N\Big\}}=X\times\cdots\times X.
\end{equation}
\end{theorem}

\begin{proof}
There
exists a $T$-invariant Borel measure which gives positive value to every non-empty open set. So, due to the syndeticity of the orbit of every point, for every $x \in X$ and every non-empty open set $U$ we have
\begin{equation}\label{E:2.14_1}
\liminf_{N\to\infty}\frac{1}{N}\sum_{n=1}^N 1_U(T^n x)>0.
\end{equation}
As we mentioned before, Theorem~\ref{T:main} implies that
\begin{equation}\label{E:2.14_2}
    \lim_{N\to\infty}\frac{1}{N}\sum_{n=1}^N T^{[p_{1,N}(n)]}f_1\cdot\ldots\cdot T^{[p_{\ell,N}(n)]}f_\ell=\prod_{i=1}^\ell \mathbb{E}(f_i|\mathcal{I}(T)).
\end{equation}
Since $\mathbb{E}(f_i|\mathcal{I}(T))= \lim_{N\to\infty}\frac{1}{N}\sum_{n=1}^N T^n f_i,$ combining \eqref{E:2.14_2} with \eqref{E:2.14_1}, we get for almost every
$x \in X$ (hence for a dense set) and every $U_1,\ldots,U_\ell$
from a given countable basis of
non-empty open sets that
\[\limsup_{N\to\infty}\frac{1}{N}\sum_{n=1}^N 1_{U_1}(T^{[p_{1,N}(n)]}x)\cdot\ldots\cdot 1_{U_\ell}(T^{[p_{\ell,N}(n)]}x)>0,\]
proving that
the set of points that satisfy \eqref{E:topol}, say $R,$ is dense. To see that $R$ is $G_\delta,$ take $\ell=1$ (the general case is analogous). Then
\[R=\left\{x\in X: \forall\;m,r\in\mathbb{N}, \exists\;N\in\mathbb{N}\;\text{and}\;1\leq n\leq N\;\text{with}\;T^{[p_{1,N}(n)]}x\in B(x_m,1/r) \right\}, \] where $\{x_m:\;m\in\mathbb{N}\}$ is a countable, dense subset of $X$ and $B(x_m,1/r)$ denotes the open ball centered at $x_m$ with radius $1/r.$ The claim now follows since
\[ R=\bigcap_{m,r\in \mathbb{N}}\bigcup_{N\in\mathbb{N}, \atop 1\leq n\leq N}T^{-[p_{1,N}(n)]}B(x_m,1/r) .\] Since $T^{[p_{i,N}(n)]}(Tx)=T(T^{[p_{i,N}(n)]}x),$ we also get the $T$-invariance of $R$.
\end{proof}


Using Zorn's lemma, we know that every dynamical system has a minimal subsystem. This fact together with Theorem~\ref{T:topol} imply the following corollary:

\begin{corollary}\label{C:topol}
For $\ell\in \mathbb{N},$ let $(p_{1,N}, \ldots, p_{\ell,N})_N$ be as in Theorem~\ref{T:main}. If $(X,T)$ is a dynamical system, then, for a non-empty and $T$-invariant set of $x\in X,$ we have
\begin{equation*}\label{E:topol'}
\begin{split}
&\overline{\Big\{(T^{[p_{1,N}(n)]}x,\ldots,T^{[p_{\ell,N}(n)]}x):\;N\in\mathbb{N}, 1\leq n\leq N\Big\}}\\
=&\overline{\{T^n x: n\in\mathbb{N}\}}\times\cdots\times \overline{\{T^n x: n\in\mathbb{N}\}}.
\end{split}
\end{equation*}
\end{corollary}

\begin{remark}
Following the method of \cite{K0} (which extended the one from \cite{FHK}), as it was adapted in \cite{KK}, the interested, and somewhat familiar to the topic, reader can state and prove the corresponding convergence results to Theorems~\ref{T:main} and ~\ref{T:mc} along prime numbers (or, for the sake of simplicity, Theorems~\ref{T:main_new} and ~\ref{T:mc_new}), together with the corresponding corollaries, as well as recurrence results along primes shifted by $\pm 1$.

While it is not trivial, this can be achieved, for the uniformity estimates, that allow one to pass from averages along natural numbers to the corresponding ones along primes, can be used for variable polynomial iterates of bounded degree (i.e., one can deal with the ``good'' and ``super nice'' variable iterates under consideration).
\end{remark}

\section{Some background material}

In this section we list some materials that will be used for the multiple average case.

\subsection{Factors}
 A \emph{homomorphism} from a system $(X,\mathcal{B}, \mu, T)$ onto a system $(Y, \mathcal{Y}, \nu,$ $S)$ is a measurable map $ \pi :X'\to Y' $, where $ X' $ is a
$ T $-invariant subset of $ X $ and $ Y' $ is an $ S $-invariant subset of $ Y $, both of full measure, such that $ \mu \circ \pi^{-1} =\nu $ and $ S\circ \pi(x)=\pi \circ T(x) $ for $ x\in X' $. When we have such a homomorphism we say that the system
$ (Y, \mathcal{Y}, \nu, S) $ is a \emph{factor} of the system $ (X,\mathcal{B}, \mu, T) $. If the factor map $ \pi :X'\to Y' $ can be chosen to be injective, then we say that the systems
$ (X,\mathcal{B}, \mu, T) $ and $ (Y, \mathcal{Y}, \nu, S) $ are \emph{isomorphic}.
 A factor can also be characterised by
$ \pi^{-1}(\mathcal{Y}) $ which is a $ T $-invariant sub-$ \sigma $-algebra of
$ \mathcal{B} $, and, conversely, any $ T $-invariant sub-$ \sigma $-algebra of
$ \mathcal{B} $ defines a factor. By abusing the terminology, we denote by the same letter the $ \sigma $-algebra $ \mathcal{Y} $ and its inverse image by $ \pi $, so, if $ (Y, \mathcal{Y}, \nu, S) $ is a factor of $ (X,\mathcal{B}, \mu, T) $, we think of
$ \mathcal{Y} $ as a sub-$ \sigma $-algebra of $ \mathcal{B}$.

\subsubsection{Seminorms}
We follow \cite{HK99} and \cite{CFH} for the inductive definition of the seminorms $\nnorm{\cdot}_k.$ More specifically, the definition that we use here follows from \cite{HK99} (in the ergodic case), \cite{CFH} (in the general case) and the use of von Neumann's mean ergodic theorem.

Let $(X,\mathcal{B},\mu,T)$ be a system and $f\in L^\infty(\mu).$  We define inductively the seminorms $\nnorm{f}_{k,\mu,T}$ (or just $\nnorm{f}_k$ if there is no room for confusion) as follows: For $k=1$ we set
\[ \nnorm{f}_{1}:= \norm{\mathbb{E}(f|\mathcal{I}(T))}_{2}.\]
Recall that the conditional
expectation $\mathbb{E}(f|\mathcal{I}(T))$ satisfies $\int \mathbb{E}(f|\mathcal{I}(T))\;d\mu=\int f\;d\mu$ and $ T\mathbb{E}(f|\mathcal{I}(T))=\mathbb{E}(Tf|\mathcal{I}(T)).$

For $k\geq 1$ we let
\[\nnorm{f}^{2^{k+1}}_{k+1}:=\lim_{N\to\infty}\frac{1}{N}\sum_{n=1}^{N}\nnorm{\bar{f}\cdot T^n f}^{2^k}_k .\]
All these limits exist and $\nnorm{\cdot}_k $ define seminorms on $ L^\infty(\mu)$ (\cite{HK99}).  Also, we remark that for all $k\in \mathbb{N}$ we have $\nnorm{f}_{k}\leq \nnorm{f}_{k+1}$ and $\nnorm{f\otimes\bar{f}}_{k,\mu\times\mu,T\times T}\leq \nnorm{f}^2_{k+1,\mu,T}.$

\subsubsection{Nilfactors}
Using the seminorms we defined above, we can construct factors $ \mathcal{Z}_k=\mathcal{Z}_k(T) $ of $ X $ characterized by:
\[ \text{ for } f\in L^\infty(\mu),\;\;\; \mathbb{E}(f|\mathcal{Z}_{k-1})=0 \text{ if and only if } \nnorm{f}_k=0. \]

The following profound fact from \cite{HK99} (see also the independent work of \cite{Ziegler}) shows that for every $ k \in \mathbb{N} $ the factor
$\mathcal{Z}_k$ has a purely algebraic structure; approximately, we can assume that it is a $k$-step nilsystem (see Subsection~\ref{Nil} below for the definitions):

\begin{theorem}[Structure Theorem, \cite{HK99,Ziegler}]\label{T:HK}
Let $(X,\mathcal{B},\mu,T)$ be an ergodic system and $ k\in\mathbb{N}$. Then the factor $\mathcal{Z}_k(T)$ is an inverse limit of $ k $-step nilsystems.\footnote{ By this we mean that there exist $T$-invariant sub-$\sigma$-algebras $\mathcal{Z}_{k,i}, i\in \mathbb{N}$, of $\mathcal{B}$ such that $\mathcal{Z}_k=\bigcup_{i\in\mathbb{N}}\mathcal{Z}_{k,i}$
and for every $i\in \mathbb{N}$, the factors induced by  the $\sigma$-algebras $\mathcal{Z}_{k,i}$ are isomorphic to $k$-step nilsystems.}
\end{theorem}



Because of this result, we call $ \mathcal{Z}_k $ the $ k $-\emph{step nilfactor} of the system. The smallest factor that is an extension of all finite step nilfactors is denoted by $ \mathcal{Z}=\mathcal{Z}(T) $, meaning, $\mathcal{Z}=\bigvee_{k\in\mathbb{N}}\mathcal{Z}_k $, and is called the \emph{nilfactor} of the system. The nilfactor
$ \mathcal{Z} $ is of particular interest because it controls the limiting behaviour in $ L^2(\mu) $ of the averages in \eqref{E:main33} and \eqref{E:mc3}.

\subsection{Nilmanifolds}\label{Nil}
Let $ G $ be a $ k $-step nilpotent Lie group, meaning $ G_{k+1}=\{e\} $ for some $ k\in \mathbb{N} $, where $ G_k=[G,G_{k-1}] $ denotes the $ k $-th commutator subgroup, and $ \Gamma $ a discrete cocompact subgroup of $ G $. The  compact homogeneous space
$ X=G/\Gamma $ is called \emph{$k$-step nilmanifold} (or \emph{nilmanifold}).
The group $ G $ acts on $ G/\Gamma $ by left translations, where the translation  by an element $ b\in G $ is given by $ T_b(g\Gamma)=(bg)\Gamma $. We denote by $ m_X $ the normalized \emph{Haar measure} on $ X,$ i.e., the unique probability measure that is invariant under the action of $ G $, and by $ \mathcal{G}/\Gamma $  the Borel $ \sigma $-algebra of $ G/\Gamma $. If $ b\in G $, we call the system $(G/\Gamma, \mathcal{G}/\Gamma,m_X,T_b)$ $k$-\emph{step nilsystem} (or \emph{nilsystem}) and the elements of  $ G $ \emph{nilrotations}.

\subsubsection{Equidistribution}\label{Ss:ud} For a connected and simply connected Lie group $G,$ let $ \exp :\mathfrak{g}\to G $ be the exponential map, where $ \mathfrak{g} $ is the Lie algebra of $ G $. For $ b \in G $ and $ s\in \mathbb{R} $ we define the element $ b^s $ of $ G $ as follows: If $ X\in \mathfrak{g} $ is such that $ \exp (X)=b $, then $ b^s=\exp(sX) $ (this is well defined since under the aforementioned assumptions $ \exp$ is a bijection).

If $ (a(n))_{n} $ is a sequence of real numbers and $ X=G/\Gamma $ is a nilmanifold with $ G $ connected and simply connected, we say that the sequence
$ (b^{a(n)}x)_{n},$ $b\in G,$ is \emph{equidistributed} in a subnilmanifold $Y$ of $ X $, if for every $ F\in C(X) $ we have
\begin{equation}\label{E:equi}
    \lim_{N\to \infty} \frac{1}{N}\sum_{n=1}^N F(b^{a(n)}x)=\int F\; d m_Y.\footnote{ If $ (a(n))_{n}\subseteq \mathbb{Z},$ we can drop the assumptions that $ G $ is connected and simply connected.}
\end{equation}


For the following claims, one can check the linear case in \cite{L} (\cite[Section~2]{L}, and in particular the theorem in \cite[Subsection~2.17]{L}, together with \cite[Theorem~2.19]{L}) which covers the $\mathbb{Z}$-actions case, and \cite{Rat} for the analogous result for $\mathbb{R}$-actions. A nilrotation $ b\in G $ is \emph{ergodic} (or \emph{acts ergodically}) on $ X $, if the sequence $ (b^n\Gamma)_{n} $ is dense in $X.$ If $ b\in G $ is ergodic, then for every
$ x\in X $ the sequence $ (b^nx)_{n} $ is equidistributed in $ X $.
The orbit closure
$Z:=\overline{(b^n\Gamma)}_{n} $ of $ b \in G$ has the structure of a nilmanifold with $ (b^n\Gamma)_{n} $ being equidistributed in $Z$. Analogously, if $ G $ is connected and simply connected, then
$W:=\overline{(b^s\Gamma)}_{s\in \mathbb{R}} $ is a nilmanifold with $ (b^s\Gamma)_{s\in \mathbb{R}} $ being equidistributed in $W$.

\subsubsection{Change of base point formula}\label{Ss:3.2.2} Let $ X=G/\Gamma $ be a nilmanifold. As mentioned before, for every $ b\in G $ the sequence $ (b^n\Gamma)_{n} $ is equidistributed in $ X_b:=\overline{\{b^n\Gamma:n\in \mathbb{N}\}} $. Using the identity $ b^ng=g(g^{-1} bg)^n $ we see that the nil-orbit $ (b^ng\Gamma)_{n} $ is equidistributed in the set $ gX_{g^{-1} bg} $. A similar formula holds when $ G $ is connected and simply connected, where we replace the $ n\in \mathbb{N} $ with $ s\in \mathbb{R} $ and the nilmanifold $ X_b $ with  $Y_b:=\overline{\{b^s\Gamma:\;s\in \mathbb{R}\}}$.

\subsubsection{Lifting argument}\label{Ss:3.2.3} Giving a topological group $G,$ we denote the connected component of its identity element, e, by $G_0.$  To assume that a nilmanifold has a representation $ G/\Gamma,$ with $ G $ connected and simply connected, one can follow for example \cite{L}.  Since all our results deal with an action on $ X $ of finitely many elements of $ G $ we can and will assume that the discrete group $ G/G_0 $ is finitely generated (see \cite[Subsection~2.1]{L}). In this case one can show (see \cite[Subsection~1.11]{L}) that $ X=G/\Gamma $ is isomorphic to a sub-nilmanifold of a nilmanifold $ \tilde{X}=\tilde{G}/\tilde{\Gamma} $, where $ \tilde{G} $ is a connected and simply connected nilpotent Lie group, with all translations from $ G $ ``represented'' in $ \tilde{G} $.\footnote{ In practice this means that for every $ F\in C(X) $, $ b\in G $ and $ x\in X $, there exists $ \tilde{F}\in C(\tilde{X}) $, $ \tilde{b}\in \tilde{G} $ and $ \tilde{x}\in \tilde{X} $, such that $ F(b^nx)=\tilde{F}(\tilde{b}^n\tilde{x}) $ for every $ n\in \mathbb{N} $.} We caution the reader that such a construction is only helpful when our working assumptions impose no restrictions on a nilrotation. Any assumption made about  $ b\in G, $ which acts on a nilmanifold $ X $, is typically lost when passing to the lifted nilmanifold
$ \tilde{X} $.



\section{Finding the characteristic factor}\label{S:4}

In this technical section we find characteristic factors for the expressions that appear in Theorems~\ref{T:main} and \ref{T:mc}. In both cases, we will show that the nilfactor is characteristic (Proposition~\ref{L:characteristic_factor} and Proposition~\ref{L:characteristic_single} respectively).


We first start with the degree $1$ case and then move on to the general one. At this point we recall the notion of a characteristic factor (adapted to our study):

\begin{definition}\label{D:CF}
For $\ell\in \mathbb{N}$ let $ (X,\mathcal{B}, \mu, T) $ be a system. The sub-$ \sigma $-algebra $ \mathcal{Y} $ of $ \mathcal{B} $ is a \emph{characteristic factor} for the variable tuple of integer-valued sequences
$ (a_{1,N},\ldots, a_{\ell,N})_N$
if it is $ T $-invariant and
\[\lim_{N\to\infty} \norm{\frac{1}{N}\sum_{n=1}^N \prod_{i=1}^\ell T^{a_{i,N}(n)}f_i-
\frac{1}{N}\sum_{n=1}^N \prod_{i=1}^\ell T^{a_{i,N}(n)}\tilde{f}_i}_{2}= 0, \]
for all $f_i\in L^\infty(\mu),$ where $\tilde{f}_i=\mathbb{E}(f_i|\mathcal{Y})$,
$1\leq i\leq \ell.$\footnote{ Equivalently, $\lim_{N\to\infty}\norm{\frac{1}{N}\sum_{n=1}^N T^{a_{1,N}(n)}f_1\cdot\ldots\cdot T^{a_{\ell,N}(n)}f_\ell}_{2}= 0$ if $\mathbb{E}(f_i|\mathcal{Y})=0$ for some $1\leq  i\leq  \ell.$}
\end{definition}

\subsection{The base case}\label{Sub:linear}
The following crucial lemma, which can be understood as a ``change of variables'' procedure, will be used in the base $\ell=1$ case for $\deg p_N=1,$ i.e., $p_N(n)=a_Nn+b_N.$ We will assume that $(b_N)_N$ is bounded, so, as such error terms do not affect our averages, we mainly have to deal with the expression $\frac{1}{N}\sum_{n=1}^N T^{[a_Nn]}f.$

\begin{lemma}\label{L:comparison}
Let $(a_N)_N\subseteq (0,+\infty)$ bounded with $(a_N\cdot N)_N$ tending increasingly to $\infty$. For any sequence $(c_N(n))_{n,N}\subseteq [0,\infty)$ we have
\[\limsup_{N\to\infty} \frac{1}{N}\sum_{n=1}^{N} c_{[a_N N]}([a_N n])\ll \limsup_{N\to\infty} \frac{1}{N}\sum_{n=1}^{N} c_N(n).\]
\end{lemma}

\begin{proof}
For a fixed $N\in \mathbb{N},$ since $(a_N)_N$ is bounded, we have the relation
\[\frac{1}{N}\sum_{n=1}^{N} c_{[a_N N]}([a_N n])\leq \left(\left[\frac{1}{a_N}\right]+1\right)\cdot a_N\cdot \frac{[a_N N]}{a_N N}\cdot \frac{1}{[a_N N]}\sum_{n=0}^{[a_N N]}c_{[a_N N]}(n).\] Since $a_N N\to \infty,$ we get that $[a_N N]/a_N N\to 1.$ Finally, using yet again that $(a_N)_N$ is bounded, we have
\[\left(\left[\frac{1}{a_N}\right]+1\right)\cdot a_N\leq a_N+1\ll 1.\] The result now follows by taking $\limsup.$
\end{proof}






Using Lemma~\ref{L:comparison}, following the argument of \cite[Lemma~5.2]{F2} we get:

\begin{lemma}\label{L:base_case_semi}
Let $(p_N)_N$ be a sequence of polynomials of degree $1$ of the form
\[p_N(n)=a_N n+b_N,\;n,N\in \mathbb{N},\] where $(a_N)_N, (b_N)_N$ are bounded sequences with $(a_N)_N\subseteq (0,+\infty)$ and $(a_N\cdot N)_N$ tending increasingly to $\infty.$ Then, for any system $(X,\mathcal{B},\mu,T)$ and $f_1\in L^\infty(\mu),$ we have
\begin{equation}\label{E:base_case_semi}
\limsup_{N\to\infty}\sup_{\norm{f_0}_\infty\leq 1}\frac{1}{N}\sum_{n=1}^{N}\left|\int f_0\cdot T^{[p_N(n)]} f_1\;d\mu\right|\ll \nnorm{f_1}_2.
\end{equation}
\end{lemma}

\begin{proof}
For every $N\in \mathbb{N}$ we choose functions $f_{0,N}$ with $\norm{f_{0,N}}_\infty\leq 1$  so that the corresponding average is $1/N$ close to its supremum $\sup_{\norm{f_0}_\infty\leq 1}.$ Inequality \eqref{E:base_case_semi}  follows if we show
\begin{equation}\label{E:base_case_semi1}
\limsup_{N\to\infty}\frac{1}{N}\sum_{n=1}^{N}\left|\int f_{0,[a_N N]}\cdot T^{[p_N(n)]} f_1\;d\mu\right|\ll \nnorm{f_1}_2.
\end{equation}
We write $[p_N(n)]=[a_N n+b_N]=[a_N n]+[b_N]+e(n,N),\; e(n,N)\in \{0,1\}.$

Let $E$ be the finite set where $([b_N]+e(n,N))_{n,N}$ takes values. We have that
 \begin{eqnarray*}
\frac{1}{N}\sum_{n=1}^{N}\left|\int f_{0,[a_N N]}\cdot T^{[p_N(n)]} f_1\;d\mu\right| & \ll & \max_{e\in E} \frac{1}{N}\sum_{n=1}^{N}\left|\int f_{0,[a_N N]}\cdot T^{[a_Nn]+e} f_1\;d\mu\right|.
\end{eqnarray*}
 Taking squares and using the Cauchy-Schwarz inequality,  the right-hand side of the previous inequality is bounded by
 \[ \max_{e\in E} \frac{1}{N}\sum_{n=1}^{N}\left|\int f_{0,[a_N N]}\cdot T^{[a_Nn]+e} f_1\;d\mu\right|^2=\max_{e\in E} \frac{1}{N}\sum_{n=1}^{N}\int F_{0,[a_N N]}\cdot S^{[a_Nn]+e} F_1\;d\tilde{\mu},\]
where $S=T\times T,$ $F_{0,[a_N N]}=f_{0,[a_N N]}\otimes \bar{f}_{0,[a_N N]},$ $F_1=f_1\otimes \bar{f}_1,$ and $\tilde{\mu}=\mu\times\mu.$  For every $e\in E$, using Lemma~\ref{L:comparison}, the $\limsup$ of the averages on the right-hand side of the
previous equality is bounded above by a constant multiple of
\[\limsup_{N\to\infty}\frac{1}{N}\sum_{n=1}^{N}\int F_{0,N}\cdot S^{n+e} F_1\;d\tilde{\mu}\leq \limsup_{N\to\infty}\norm{\frac{1}{N}\sum_{n=1}^{N} S^{n+e} F_1}_{L^2(\tilde{\mu})},\] where the last inequality follows by Cauchy-Schwarz and the fact that $\norm{F_{0,N}}_\infty\leq 1$. Using von Neumann's mean ergodic theorem, the last term is equal to
\[\norm{\mathbb{E}(S^e F_1|\mathcal{I}(S))}_{L^2(\tilde{\mu})}=\norm{\mathbb{E}(F_1|\mathcal{I}(S))}_{L^2(\tilde{\mu})}\leq \nnorm{f_1}_2^2,\] where we used the fact that $S$ is measure preserving, the definition of the seminorms $\nnorm{\cdot},$ and the relationship between the $k$-th seminorm
of the tensor product and the $k + 1$ seminorm on the base space. Inequality \eqref{E:base_case_semi1} now follows by removing the squares.
\end{proof}

\begin{remark}
Lemma~\ref{L:base_case_semi} holds also for sequences $(a_N)_N\subseteq (-\infty,0)$ with $(a_N\cdot N)_N$ tending decreasingly to $-\infty.$

\emph{Indeed, In this case we write
\[[p_N(n)]=-[-a_N n]+[b_N]+e(n,N),\;\;e(n,N)\in \{-1,0\},\]
so,
\[\frac{1}{N}\sum_{n=1}^{N}\left|\int f_0\cdot T^{[p_N(n)]} f_1\;d\mu\right|\ll \max_{e\in E}\frac{1}{N}\sum_{n=1}^{N}\left|\int f_0\cdot T^{-[-a_N n]} (T^e f_1)\;d\mu\right|,\]
where $E$ is a finite subset of integers. Since $-a_N>0$ and $(-a_N\cdot N)_N$ tends increasingly to $\infty,$ we get the conclusion by the previous lemma (working with $T^{-1}$ instead of $T$).}\footnote{ We note that since in Theorems~\ref{T:main} and ~\ref{T:mc} we assume that the transformation $T$ or $T^{-1}$
is ergodic, then the seminorms
taken with respect to either of those transformations coincide.}
\end{remark}

For multiple terms, we use the following variant of the classical van der Corput trick:

\begin{lemma}[Lemma~4.6, \cite{F2}]\label{L:vdc}
Let $(v_{N,n})_{N,n}$ be a bounded sequence in a Hilbert space. Then
\[\limsup_{N\to\infty}\norm{\frac{1}{N}\sum_{n=1}^N v_{N,n}}^2\leq 4 \limsup_{H\to\infty}\frac{1}{H}\sum_{h=1}^H \limsup_{N\to\infty}\left|\frac{1}{N}\sum_{n=1}^N \langle v_{N,n+h},v_{N,n}\rangle\right|.\]
\end{lemma}

We will now demonstrate the main idea behind the generalization of Lemma~\ref{L:base_case_semi}, for which we follow  \cite[Proposition~5.3, Case~1]{F2}. In that statement, to show
\[\limsup_{N\to\infty}\sup_{\norm{f_0}_\infty,\norm{f_i}_\infty\leq 1}\frac{1}{N}\sum_{n=1}^N\left|\int f_0\cdot T^{[a_1 n]}f_1\cdot T^{[a_2 n]}f_2\right|\ll \nnorm{f_j}_4,\] where $(i,j)=(1,2)$ or $(2,1),$ one uses Lemma~\ref{L:vdc}, compose with, say, $-[a_1 n],$ and gets the terms (notice that we keep the $h$-term in the first difference even though it is bounded)
\[[a_1 (n+h)]-[a_1 n]\approx [a_1 h],\]
\[[a_2 (n+h)]-[a_1 n]\approx [(a_2-a_1)n],\;\;\text{and}\]
\[[a_2 n]-[a_1 n]\approx [(a_2-a_1)n],\footnote{ Here, by ``$\approx$'', we mean ``equality modulo bounded error terms''.}\] so, after grouping the last two terms together, using the first one as constant (since it only depends on $h$--the average along which is crucial for the argument and is taken at the very end), one can use the base $\ell=1$ case. This is also how the inductive step works in the proof of the general $\ell\in \mathbb{N}$ case.

The variable case is more complicated to deal with. We demonstrate the main idea behind it by considering Example~\ref{Ex:1}, i.e., $p_{1,N}(n)=a_{1,N}n$  and $p_{2,N}(n)=a_{2,N}n,$  where $a_{1,N}=1/N^a$ and $a_{2,N}=1/N^b$ for $0<a<b<1$. The previous approach cannot be imitated, as, for example,
\[[a_{1,N}(n+h)]-[a_{1,N}n]\approx [a_{1,N}h]\] is in general a variable term and we cannot proceed with the same argument. What we do instead is to transform the iterates in the initial sum to the following:
\[(0,[a_{1,N}n],[a_{2,N}n])\approx \left(0,[a_{1,N}n],\left[\frac{a_{2,N}}{a_{1,N}}[a_{1,N}n]\right]\right)\xrightarrow[\text{}]{\text{Lemma~\ref{L:comparison}}}\left(0,n,\left[\frac{a_{2,N}}{a_{1,N}}n\right]\right),\] and then we use Lemma~\ref{L:vdc} (i.e., change of variables) to bound, eventually, everything by $\nnorm{f_1}_4.$ (To use Lemma~\ref{L:comparison} note the crucial fact that $(a_{2,N}/a_{1,N})_N$ is bounded.)
Additionally, to bound our expression by $\nnorm{f_2}_4,$ the previous argument needs an additional twist to work since the quantity $(a_{1,N}/a_{2,N})_N$ is unbounded. What we do in this case is to compose with $-[a_{2,N}n]$ to get
\begin{eqnarray*}(0,[a_{1,N}n],[a_{2,N}n]) & \approx & ([-a_{2,N}n],[(a_{1,N}-a_{2,N})n],0)
\\& \approx & \left(\left[\frac{-a_{2,N}}{a_{1,N}-a_{2,N}}[(a_{1,N}-a_{2,N})n]\right],[(a_{1,N}-a_{2,N})n],0\right)
\\&\rightarrow  & \left(\left[\frac{-a_{2,N}}{a_{1,N}-a_{2,N}}n\right],n,0\right),
\end{eqnarray*} where we used the change of variables.
As $(a_{2,N}/(a_{1,N}-a_{2,N}))_N$ is bounded, we can now finish the argument as before.

The previous discussion, naturally leads to the following assumption on the leading coefficients of the linear (variable) polynomials:

\begin{definition}\label{D:R_property}
A sequence of real numbers $(a_N)_N$ has the \emph{$R_1$-property} if
\begin{itemize}
    \item[(i)] it is bounded; and

    \item[(ii)] $(a_{N})_N\subseteq (0,+\infty)$ or $(-\infty,0)$ and $(|a_{N}|\cdot N)_N$ tends increasingly to $+\infty.$
\end{itemize}

For $\ell\in \mathbb{N}$ the sequences $\{(a_{i,N})_N:\;1\leq i\leq \ell\}$ have the \emph{$R_\ell$-property} if for all $1\leq i\leq \ell$:
\begin{itemize}
    \item[(i)] $(a_{i,N})_N$ has the $R_1$-property; and

    \item[(ii)]  at least one of the following three properties holds:\end{itemize}

\noindent (a) $\exists \; 1\leq j_0\neq i\leq \ell$ such that $\bigg\{\left(\frac{a_{j_0,N}-a_{j,N}}{a_{i,N}}\right)_N:\;1\leq$ $ j\neq j_0\leq \ell\bigg\}$ have the $R_{\ell-1}$-property.

\noindent  (b)  $\exists \; 1\leq j_0\neq i\leq \ell$ such that $(a_{i,N}-a_{j_0,N})_N$ has the $R_1$-property and the sequences $\bigg\{\left(\frac{a_{j,N}}{a_{i,N}-a_{j_0,N}}\right)_N:\;1\leq j\neq j_0\leq \ell\bigg\}$ have the $R_{\ell-1}$-property.

\noindent   (c)  $\exists \; 1\leq j_0\neq i\leq \ell$ such that $(a_{i,N}-a_{j_0,N})_N$ has the $R_1$-property and $1\leq k_0\neq j_0,$ $i\leq \ell$ such that  $\bigg\{\left(-\frac{a_{k_0,N}}{a_{i,N}-a_{j_0,N}}\right)_N,$ $\left(\frac{a_{j,N}-a_{k_0,N}}{a_{i,N}-a_{j_0,N}}\right)_N:\;1\leq j\neq k_0, j_0\leq \ell\bigg\}$ have the $R_{\ell-1}$-property.
\end{definition}

\begin{remark}\label{R:Ex1}
{The polynomial family of Example~\ref{Ex:1}, i.e., $p_{1,N}(n)=n/N^a,$ $p_{2,N}(n)=n/N^b,$ $n, N\in \mathbb{N},$ where $0<a<b<1,$ has the $R_2$-property.}

{Indeed,
skipping the trivial calculations, both sequences $(1/N^a)_N,$ $(1/N^b)_N$ have the $R_1$-property and for $i=1$ we have the $(ii)$ $(a)$ case, while for $i=2$ the $(ii)$ $(b)$ case.}
\end{remark}

We are now ready to extend Lemma~\ref{L:base_case_semi} to multiple terms along polynomials of degree~$1,$ following the main idea of \cite[Proposition~5.3, Case~1]{F2}:

\begin{proposition}\label{P:Linear}
Let $(p_{1,N})_N,\ldots,(p_{\ell,N})_N$ be polynomial sequences of degree $1$ of the form
\[p_{i,N}(n)=a_{i,N}n+b_{i,N},\;\;n, N \in \mathbb{N}, 1\leq i\leq \ell,\] where the sequences $(a_{i,N})_N,$ $1\leq i\leq \ell,$ have the $R_\ell$-property and $(b_{i,N})_N,$ $1\leq i\leq \ell,$ are bounded. Then, for every $f_1\in L^\infty(\mu),$ we have
\begin{equation}\label{E:Linear}
\limsup_{N\to\infty}\sup_{\norm{f_0}_\infty,\norm{f_2}_\infty,\ldots,\norm{f_\ell}_\infty\leq 1}\frac{1}{N}\sum_{n=1}^N \left|\int f_0\cdot \prod_{i=1}^\ell T^{[p_{i,N}(n)]}f_i\;d\mu\right|\ll \nnorm{f_1}_{2\ell}.\footnote{ The implicit constant in \eqref{E:Linear} depends on the bounds of the coefficients and the number of transformations $\ell.$ Note that, because of symmetry, we also have the respective estimates for $2\leq i\leq \ell$ with $f_i$ in place of $f_1.$}
\end{equation}
\end{proposition}

\begin{proof}

We use induction on $\ell.$ The base case, $\ell=1,$ follows from Lemma~\ref{L:base_case_semi}.
We assume that $\ell\geq 2$ and that the statement holds for $\ell-1.$

\medskip

\noindent{\bf{Case 1:}} For $i=1,$ the property (ii) (a) from the Definition~\ref{D:R_property} holds.
\begin{equation}\label{E:Linear1}
	\begin{split}
	&\quad \frac{1}{N}\sum_{n=1}^N \left|\int f_0\cdot T^{[a_{1,N}n+b_{1,N}]}f_1\cdot\ldots\cdot T^{[a_{\ell,N}n+b_{\ell,N}]}f_\ell\;d\mu\right|
	\\&= \frac{1}{N}\sum_{n=1}^N \left|\int f_0\cdot T^{[a_{1,N}n]+e_1(n,N)}f_1\cdot\prod_{i=2}^\ell T^{\left[\frac{a_{i,N}}{a_{1,N}}[a_{1,N}n]\right]+e_i(n,N)}f_i\;d\mu\right|
	\\&\ll \max_{e_1,\ldots,e_\ell\in E}\frac{1}{N}\sum_{n=1}^N \left|\int f_0\cdot T^{[a_{1,N}n]}(T^{e_1}f_1)\cdot\prod_{i=2}^\ell T^{\left[\frac{a_{i,N}}{a_{1,N}}[a_{1,N}n]\right]}(T^{e_i}f_i)\;d\mu\right|,
	\end{split}
	\end{equation}
where $E$ is a finite subset of integers (the error terms $e_i(n,N)$, as $(b_{i,N})_N,$ and $(a_{i,N}/a_{1,N})_N$ are bounded for $1\leq i\leq \ell,$ take finitely many values).

For every $N\in \mathbb{N}$ we now choose functions $f_{i,N}$ with $\norm{f_{i,N}}_\infty\leq 1$ for $i\in \{0,2,\ldots,\ell\},$ so that the last term in \eqref{E:Linear1} is $1/N$ close to $\sup_{\norm{f_0}_\infty,\norm{f_2}_\infty,\ldots,\norm{f_\ell}_\infty\leq 1}.$ Using the Cauchy-Schwarz inequality and the fact that $|a_{1,N}|\cdot N\to\infty,$ we have that \eqref{E:Linear} follows if we show, for each choice of  $e_1,\ldots,e_\ell\in E,$ that
\begin{equation}\label{E:Linear2}
\limsup_{N\to\infty}\frac{1}{N}\sum_{n=1}^N \left|\int f_{0,[a_{1,N}N]}\cdot T^{[a_{1,N}n]}(T^{e_1}f_1)\cdot\prod_{i=2}^\ell T^{\left[\frac{a_{i,N}}{a_{1,N}}[a_{1,N}n]\right]}(T^{e_i}f_{i,[a_{1,N}N]})\;d\mu\right|^2
	\end{equation} is bounded above by a constant multiple of $\nnorm{f_1}^2_{2\ell}.$
Using Lemma~\ref{L:comparison} it suffices to show
\begin{equation}\label{E:Linear3}
\limsup_{N\to\infty}\frac{1}{N}\sum_{n=1}^N \left|\int f_{0,N}\cdot T^{n}(T^{e_1}f_1)\cdot\prod_{i=2}^\ell T^{\left[\frac{a_{i,N}}{a_{1,N}}n\right]}(T^{e_i}f_{i,N})\;d\mu\right|^2\ll \nnorm{f_1}^2_{2\ell}.
	\end{equation}
The left-hand side of \eqref{E:Linear3} is equal to
\begin{equation*}\label{E:Linear4}
A:=\limsup_{N\to\infty}\frac{1}{N}\sum_{n=1}^N \int F_{0,N}\cdot S^{n}(S^{e_1}F_1)\cdot\prod_{i=2}^\ell S^{\left[\frac{a_{i,N}}{a_{1,N}}n\right]}(S^{e_i}F_{i,N})\;d\tilde{\mu},
	\end{equation*}
where $S=T\times T,$ $F_1=f_1\otimes\bar{f}_1,$ $F_{i,N}=f_{i,N}\otimes \bar{f}_{i,N},$ $i=0,2,\ldots,\ell,$ and $\tilde{\mu}=\mu\times\mu.$
Using the Cauchy-Schwarz inequality and Lemma~\ref{L:vdc}, we have that
\begin{equation*}\label{E:Linear5}
|A|^2\ll \limsup_{H\to\infty}\frac{1}{H}\sum_{h=1}^H A_h,
\end{equation*}
where
\begin{equation*}
	\begin{split}
	&\quad A_h:=\limsup_{N\to\infty}\frac{1}{N}\sum_{n=1}^N \bigg|\int S^{n+h}(S^{e_1}F_1)\cdot\prod_{i=2}^\ell S^{\left[\frac{a_{i,N}}{a_{1,N}}(n+h)\right]}(S^{e_i}F_{i,N})\
	\\
&  \quad\quad\quad\quad\quad\quad\quad\quad\quad\quad\quad\quad\quad\quad\quad\quad\quad\quad \cdot S^{n}(S^{e_1}\overline{F}_1)\cdot\prod_{i=2}^\ell S^{\left[\frac{a_{i,N}}{a_{1,N}}n\right]}(S^{e_i}\overline{F}_{i,N}) \;d\tilde{\mu} \bigg|.
	\end{split}
    \end{equation*}
Precomposing with the term $S^{-n}$ we get
\begin{equation}\label{E:change_to_h}
	\begin{split}
	&\quad A_h=\limsup_{N\to\infty}\frac{1}{N}\sum_{n=1}^N \bigg|\int S^{h}(S^{e_1}F_1)\cdot S^{e_1}\overline{F}_1\\
&  \quad\quad\quad\quad\quad\quad\quad\quad\quad\quad\quad \cdot\prod_{i=2}^\ell S^{\left[\left(\frac{a_{i,N}}{a_{1,N}}-1\right)n\right]+\left[\frac{a_{i,N}}{a_{1,N}}h\right]+e_{j}(n,h,N)}(S^{e_i}F_{i,N})
	\\
&  \quad\quad\quad\quad\quad\quad\quad\quad\quad\quad\quad\quad\quad\quad\quad\quad\quad \cdot \prod_{i=2}^\ell S^{\left[\left(\frac{a_{i,N}}{a_{1,N}}-1\right)n\right]+\tilde{e}_{j}(n,h,N)}(S^{e_i}\overline{F}_{i,N}) \;d\tilde{\mu}\bigg|
\\& \quad \quad= \limsup_{N\to\infty}\frac{1}{N}\sum_{n=1}^N \bigg|\int F_{1,h}\cdot\prod_{i=2}^\ell S^{\left[\left(\frac{a_{i,N}}{a_{1,N}}-1\right)n\right]}F_{i,h,n,N} \;d\tilde{\mu}\bigg|,
	\end{split}
    \end{equation}
where $F_{1,h}=S^{h}(S^{e_1}F_1)\cdot S^{e_1}\overline{F}_1,$ $F_{i,h,n,N}=S^{\left[\frac{a_{i,N}}{a_{1,N}}h\right]+e_{j}(n,h,N)}(S^{e_i}F_{i,N})\cdot S^{\tilde{e}_{j}(n,h,N)}\\($ $S^{e_i}\overline{F}_{i,N}),$ $2\leq i\leq \ell.$
 Using the hypothesis, for $i=1,$ there exists $2\leq j_0\leq \ell$ such that the sequences $\bigg\{\left(\frac{a_{j_0,N}-a_{j,N}}{a_{1,N}}\right)_N:\;1\leq j\neq j_0\leq \ell\bigg\}$ have the $R_{\ell-1}$-property. Precomposing with $S^{-\left[\left(\frac{a_{j_0,N}}{a_{1,N}}-1\right)n\right]}$ in the right-hand side of \eqref{E:change_to_h} we have that
 \begin{equation*}
	\begin{split}
	&\quad A_h=\limsup_{N\to\infty}\frac{1}{N}\sum_{n=1}^N \bigg|\int F_{1,h}\cdot\prod_{i=2}^\ell S^{\left[\left(\frac{a_{i,N}}{a_{1,N}}-1\right)n\right]}F_{i,h,n,N} \;d\tilde{\mu}\bigg|
\\&\quad\quad = \limsup_{N\to\infty}\frac{1}{N}\sum_{n=1}^N \bigg|\int F_{j_0,h,n,N}\cdot S^{\left[-\left(\frac{a_{j_0,N}}{a_{1,N}}-1\right)n\right]+e'_1(n,N)}F_{1,h}
\\&\quad\quad\quad\quad\quad\quad\quad\quad\quad\quad\quad\quad \quad\quad\quad\quad \cdot \prod_{2\leq i\neq j_0\leq \ell} S^{\left[-\left(\frac{a_{j_0,N}-a_{i,N}}{a_{1,N}}\right)n\right]}\tilde{F}_{i,h,n,N} \;d\tilde{\mu}\bigg|,
	\end{split}
    \end{equation*}
where $\tilde{F}_{i,h,n,N}=S^{e'_i(n,N)}F_{i,h,n,N}$ for some error terms $e'_i(n,N)\in \{0,1\}.$

As we previously highlighted, for every fixed $N,$ we can partition the set of integers so that $e_1'(n,N)$ is constant. So, fixing $e_1'\in \{0,1\},$ using the induction hypothesis, we have
\begin{equation*}
	\begin{split}
	&\quad A_h\ll \limsup_{N\to\infty}\sup_{\norm{F_0}_\infty,\norm{F_2}_\infty,\ldots,\norm{F_\ell}_\infty\leq 1}\frac{1}{N}\sum_{n=1}^N \bigg|\int F_0\cdot S^{\left[-\left(\frac{a_{j_0,N}}{a_{1,N}}-1\right)n\right]}(S^{e'_1}F_{1,h})
	\\
&  \quad\quad\quad\quad\quad\quad\quad\quad\quad\quad\quad\quad\quad\quad\quad\quad\quad\quad\quad \cdot\prod_{2\leq i\neq j_0\leq \ell} S^{\left[-\left(\frac{a_{j_0,N}-a_{i,N}}{a_{1,N}}\right)n\right]}F_{i} \;d\tilde{\mu}\bigg|
\\& \quad \quad \ll \nnorm{S^{e'_1}F_{1,h}}_{2(\ell-1)}=\nnorm{F_{1,h}}_{2(\ell-1)} =  \nnorm{S^{h}(S^{e_1}F_1)\cdot S^{e_1}\overline{F}_1}_{2(\ell-1)}
\\& \quad \quad = \nnorm{(T^{h+e_1}f_1\cdot T^{e_1}\overline{f}_1)\otimes \overline{(T^{h+e_1}f_1\cdot T^{e_1}\overline{f}_1)}}_{2(\ell-1)}\leq \nnorm{T^{h+e_1}f_1\cdot T^{e_1}\overline{f}_1}^2_{2\ell-1}.
	\end{split}
    \end{equation*}

So, using H\"older inequality and the definition of the seminorms $\nnorm{\cdot},$ we have
\begin{eqnarray*}
|A|^2 & \ll & \limsup_{H\to\infty}\frac{1}{H}\sum_{h=1}^H A_h \ll \limsup_{H\to\infty}\frac{1}{H}\sum_{h=1}^H \nnorm{T^{h+e_1}f_1\cdot T^{e_1}\overline{f}_1}^2_{2\ell-1}\\
& \leq & \limsup_{H\to\infty}\left(\frac{1}{H}\sum_{h=1}^H \nnorm{T^{h+e_1}f_1\cdot T^{e_1}\overline{f}_1}^{2^{2\ell-1}}_{2\ell-1} \right)^{1/2^{2(\ell-1)}}=\nnorm{T^{e_1}f_1}^4_{2\ell}=\nnorm{f_1}^4_{2\ell},
\end{eqnarray*}
hence, \eqref{E:Linear2} is bounded above by a constant multiple of $\nnorm{f_1}^2_{2\ell}$ as was to be shown.

\medskip

\noindent{\bf{Cases 2 \& 3:}} For $i=1,$ we either have property (ii) (b) or (ii) (c) in Definition~\ref{D:R_property}.

Here we will skip the details already outlined in Case 1.
If $2\leq j_0\leq \ell$ is the integer guaranteed by Definition~\ref{D:R_property}, the integrand in the last part of equation \eqref{E:Linear1} will become (setting, without loss, $e_i=0$)
\begin{equation*}
\begin{split}
&\quad f_{j_0}\cdot T^{[(a_{1,N}-a_{j_0,N})n]}f_1\cdot T^{\left[-\frac{a_{j_0,N}}{a_{1,N}-a_{j_0,N}}[(a_{1,N}-a_{j_0,N})n]\right]}f_0\\
&\quad\quad\quad\quad\quad\quad\quad\quad\quad\quad\quad\quad\quad\quad\quad\quad\quad\quad\cdot\prod_{2\leq j\neq j_0\leq \ell}T^{\left[\frac{a_{j,N}-a_{j_0,N}}{a_{1,N}-a_{j_0,N}}[(a_{1,N}-a_{j_0,N})n]\right]}f_j
\end{split}
\end{equation*}
and the one in equation \eqref{E:change_to_h}
\[ F_{1,h}\cdot S^{\left[\left(\frac{-a_{j_0,N}}{a_{1,N}-a_{j_0,N}}-1\right)n\right]}F_{0,h,n,N}\cdot \prod_{2\leq j\neq j_0\leq \ell} S^{\left[\left(\frac{a_{j,N}-a_{j_0,N}}{a_{1,N}-a_{j_0,N}}-1\right)n\right]}F_{i,h,n,N}.\]
Precomposing with the term $S^{-\left[\left(\frac{-a_{j_0,N}}{a_{1,N}-a_{j_0,N}}-1\right)n\right]}=S^{-\left[\frac{a_{1,N}}{a_{1,N}-a_{j_0,N}}n\right]}$ (for Case 2) and with $S^{-\left[\left(\frac{a_{k_0,N}-a_{j_0,N}}{a_{1,N}-a_{j_0,N}}-1\right)n\right]} = S^{-\left[\frac{a_{k_0,N}-a_{1,N}}{a_{1,N}-a_{j_0,N}}n\right]}$ (for Case 3--where $2\leq k_0\neq j_0\leq \ell$ is the one guaranteed by Definition~\ref{D:R_property}), we can continue (using the induction hypothesis) and finish the argument as in Case 1. The proof of the statement is now complete.
\end{proof}

\begin{remark}\label{R:uniform}
To the best of our knowledge, when we deal with norm convergence of averages of (non-variable) polynomial iterates, we can always replace the conventional Ces\`aro averages, i.e., $\lim_{N\to\infty}\frac{1}{N}\sum_{n=1}^N,$ with the corresponding uniform ones, i.e., $\lim_{N-M\to\infty}\frac{1}{N-M}\sum_{n=M}^{N-1}.$  Our method though, exactly because of the choice of functions $f_{i,[a_{1,N}N]}$ (to go from equation \eqref{E:base_case_semi} to \eqref{E:base_case_semi1} and from equation \eqref{E:Linear1} to \eqref{E:Linear2}), cannot guarantee the corresponding uniform results.
\end{remark}

\subsection{The general case}\label{S:general_case}

We start by recalling (see, for example, \cite{Be} and \cite{F2}) the definition of the degree and type of a polynomial family that we will adapt in our study:

\begin{definition}\label{D:type-degree}
For $\ell\in \mathbb{N}$ let $\mathcal{P}=\{p_{1},\ldots,p_{\ell}\}$ be a family of non-constant real polynomials. We denote with $\deg(\mathcal{P})$ the maximum degree of the polynomials $p_i$'s and we call it \emph{degree} of $\mathcal{P}.$
If $w_i$ denotes the number of distinct leading coefficients of polynomials from $\mathcal{P}$ of degree $i$ and $d=\deg(\mathcal{P}),$ then the vector $(d,w_d,\ldots,w_1)$ is the \emph{type} of $\mathcal{P}.$ We order all the possible type vectors lexicographically.\footnote{ I.e., $(d,w_d,\ldots,w_1)>(d',w'_d,\ldots,w'_1)$ iff, reading from left to right, the first instance where the two vectors disagree the coordinate of the first vector is greater than that of the second one.}
\end{definition}

In order to reduce the complexity (i.e., the type) of a polynomial family, one has to use the classic PET (i.e., Polynomial Exhaustion Technique) induction.

At this point we remind the reader that the real polynomials $p_{1},\ldots,p_{\ell}$ are called \emph{essentially distinct} if they are, together with their pairwise differences, non-constant. Given such a family of polynomials $\mathcal{P}=\{p_{1},\ldots,p_{\ell}\},$ $p\in \mathcal{P}$ and $h\in \mathbb{N},$ the \emph{van der Corput operation (vdC-operation)}, acting on $\mathcal{P},$ gives the family
\[\mathcal{P}(p,h):=\{p_1(t+h)-p(t),\ldots,p_\ell(t+h)-p(t),p_1(t)-p(t),\ldots,p_\ell(t)-p(t)\},\footnote{ Notice that if $\mathcal{P}$ lists the polynomial iterates in the expression  $T^{[p_1(n)]}f_1\cdot\ldots\cdot T^{[p_\ell(n)]}f_\ell,$ then $\mathcal{P}(p,h)$ lists the respective iterates (modulo error terms) after using Lemma~\ref{L:vdc} and precomposing with the iterate $p(n).$}\] where we then remove all the terms that are bounded\footnote{ This is justified with the use of the Cauchy-Schwarz inequality.} and we group the ones of degree $1$ with bounded difference (i.e., of the same leading coefficient), thus obtaining a new family of essentially distinct polynomials.

The following lemma states that there exists a choice of a polynomial in a family of essentially distinct polynomials, via which the vdC-operation reduces its type:

\begin{lemma}[Lemma~4.5, \cite{F2}]\label{L:type}
Let $\ell\in \mathbb{N}$ and $\mathcal{P}=\{p_1,\ldots,p_\ell\}$ be a family of essentially distinct polynomials with $\deg(\mathcal{P})=\deg(p_1)\geq 2.$ Then there exists $p\in \mathcal{P}$ (of minimum degree in the polynomial family) such that for every large $h$ the family $\mathcal{P}(p,h)$ has type smaller than that of $\mathcal{P},$ and $\deg(\mathcal{P}(p,h))=\deg(p_1(t+h)-p(t)).$
\end{lemma}

What is crucial for us is that every decreasing sequence of types is eventually (after finitely many steps) stationary and that, by using the previous lemma, there is a point at which all the polynomials have degree $1$. Also, by its definition, the vdC-operation preserves the essential distinctness property.





We will deal with sequences of families of real polynomials, $(\mathcal{P}_N)_N,$ where $\mathcal{P}_N=\{p_{1,N},\ldots,p_{\ell,N}\},$ $N\in \mathbb{N},$ 
that, for large $N$, 
has type independent of $N,$ to be able to use the facts that we just mentioned. Abusing the notation, we write $(\mathcal{P}_N)_N=(p_{1,N},\ldots,p_{\ell,N})_N$.




Next we define the subclass of variable polynomials that we will deal with.

\begin{definition}\label{D:super_nice}
For $\ell\in \mathbb{N}$ let $(\mathcal{P}_N)_N=(p_{1,N},\ldots,p_{\ell,N})_N$ be a sequence of $\ell$-tuples of real polynomials with bounded coefficients. We say that $(\mathcal{P}_N)_N$ is \emph{super nice} if, for every (large enough) $N\in \mathbb{N}$:
\begin{itemize}
    \item[(i)]  the polynomials $p_{i,N}$ and, for all $i\neq j,$ $p_{i,N}-p_{j,N}$ are non-constant and their degrees are independent of $N$;





    \medskip

    \item[(ii)] after performing, if needed, (finitely many) vdC-operations to $(\mathcal{P}_N)_N$ to obtain only polynomials of degree $1,$ say $k\equiv k((\mathcal{P}_N)_N)$ many, the leading coefficients (for large enough $h_i$'s--from the vdC-operations) have the $R_k$-property; and

    \medskip

    \item[(ii)$'$] if $\deg((p_{i_0,N})_N)=\deg((\mathcal{P}_N)_N),$ then (ii) holds for the polynomial sequence $(\mathcal{P}'_N)_N$ $:=(p_{1,N}-p_{i_0,N},\ldots,p_{i_0-1,N}-p_{i_0,N},-p_{i_0,N},p_{i_0+1,N}-p_{i_0,N},\ldots,p_{\ell,N}-p_{i_0,N})_N.$
\end{itemize}
\end{definition}


\begin{remark}\label{R:super}
$(1)$ It is not clear to us whether $(ii)$ implies $(ii)'$.

Consider for example the sequence of polynomials $(\mathcal{P}_N)_N=(p_{1,N},p_{2,N})_N,$ where $p_{1,N}(n)$ $=-a_N n^2-b_N n$ and $p_{2,N}(n)=(a_N-b_N)n.$  After performing the vdC-operation twice we get the triple $\{-2a_N(h'+h)n,$  $-2a_N h n, -2a_N h' n\},$ while for the sequence $(\mathcal{P}_N')_N=(-p_{1,N},p_{2,N}-p_{1,N})_N,$ after a single use of the vdC-operation we get $\{(a_N-b_N)n, 2a_Nh n,$ $((2h+1)a_N-b_N)n\}$. So, in the second case we have to impose assumptions on both $(a_N)_N,$ $(b_N)_N,$ while in the first one only on $(a_N)_N.$

\medskip

$(2)$ The degree and type of every super nice sequence, together with the integer $k$ in $(ii)$ (and, analogously, in $(ii)'$ as well), are independent of $N.$

\medskip

$(3)$ Every family of essentially distinct polynomials that
does not depend on N is super nice.

Indeed, since $(i)$ is immediate, we are showing $(ii)$ ($(ii)'$ follows by the same argument). As it was mentioned before, the vdC-operation preserves the essential distinctness property, hence, all the $k$ linear polynomial will have distinct leading coefficients, which, as they are independent of $N,$ will  have the $\mathcal{R}_k$-property.

\medskip

$(4)$ The set of super nice variable polynomial sequences is non-empty. Actually, the $\ell$-tuple $(p_{1,N},\ldots,p_{\ell,N})_N,$ where $p_{i,N}(n)=n^i/N^a,$ $1\leq i\leq \ell,$ $N,n\in \mathbb{N},$ and $0<a<1,$ from Example~\ref{Ex:2}, is super nice (see Lemma~\ref{L:are_super_nice} below for a more general statement).

Indeed, as the variable part of the coefficients of the polynomials, after applying vdC-operations, is the same for all terms (and equal to $1/N^a$), at each step  we have that the ratios of the coefficients are independent of $N$, hence we have all the required properties.


\medskip

$(5)$  
Even though the number $k$ of degree $1$ terms (that appears in $(ii)$ and $(ii)'$) 
is not a priori known, when we have a single variable polynomial sequence \[p_N(n)=a_{d,N}n^d+\ldots+a_{1,N}n+a_{0,N},\] where $(a_{d,N})_N$ has the $R_1$-property and all $(a_{i,N})_N,$ $0\leq i\leq d,$ are bounded, we have that for all $\ell\in \mathbb{N},$ $(\mathcal{P}_N)_N=(p_N,2p_N,\ldots,\ell p_N)_N$ is super nice.

It suffices to show only (ii).  We start with $(\ell p_N(n),\ldots,$ $p_N(n))$ and use the vdC-operation which leads to differences of polynomials.\footnote{ Notice that, for every $h\in \mathbb{N},$  $\Delta_1(p(n);h):=p(n+h)-p(n)$  reduces the degree of $p$ by $1$.} Precomposing with $-p_N(n)$ we get the family of polynomials \[(\ell-1)p_{N}(n+h_1)+p_{N}(n+h_1)-p_N(n),\ldots,p_N(n+h_1)-p_N(n),(\ell-1)p_N(n),\ldots,p_N(n).\]
In the next iteration of the vdC-operation we precompose with $-p_N(n+h_1)+p_N(n)$, and then $-(p_N(n+h_1+h_2)-p_N(n+h_2))+(p_N(n+h_1)-p_N(n))$ (i.e., polynomials of minimum degree at each step). We will keep track of the leading coefficients of polynomials of maximum degree at each step;\footnote{ Here, as we only have distinct non-zero multiples of the same polynomial it is not hard to do so; for more general coefficient tracking methods see \cite{DFMKS,DKS2}.} we have the following cases in this procedure:

$\bullet$ The polynomial that is chosen according to Lemma~\ref{L:type} has degree strictly less than the one of the polynomial of maximum degree (e.g., this happens in the second iteration of the vdC-operation). In this case the leading coefficient of the latter polynomial doesn't change.

$\bullet$ The polynomial that is chosen according to Lemma~\ref{L:type}, say $q_N,$ has degree, say $D,$ equal to the one of the polynomial of maximum degree (hence all the polynomials have the same degree $D$--this is the case in the first application of the vdC-operation). Here, because of the nature of the (essentially distinct) iterates, the leading coefficients will be multiples of the leading coefficient of $q_N.$ The scheme will continue by picking for the next step the polynomial $q_N(n+h)-q_N(n)$ (for the corresponding shift $h\in\mathbb{N}$ with leading coefficient $D$ times the leading coefficient of $q_N$) which is of minimum degree.

Continuing the procedure, we eventually arrive at, say $k$ many, degree $1$ iterates with distinct leading coefficients (because of the essential distinctness property), which are all multiples of $d!\cdot a_{d,N}$ (i.e., the coefficient of $p_N^{(d-1)}(n)$). As all the iterated ratios of these coefficients are independent of $a_{d,N}$ and non-zero, we get that they satisfy the $R_k$-property.

\medskip

$(6)$ If $(\mathcal{P}_N)_N$ is super nice, then $(\mathcal{P}'_N)_N$ is super nice too.

Looking at Property $(i)$ for $(\mathcal{P}_N)_N,$ we have that each polynomial (sequence) in $(\mathcal{P}'_N)_N$ is non-constant and has degree independent of $N,$ equal to $\deg(p_{i_0,N}-p_{1,N}).$ If we let \[q_{i,N}:= \begin{cases}
      p_{i,N}-p_{i_0,N} & i\neq i_0 \\
      -p_{i_0,N} & i=i_0
   \end{cases},\;\text{then we have}\;
 q_{i,N}-q_{j,N}=\begin{cases}
      p_{i,N}-p_{j,N} & i,j\neq i_0 \\
      p_{i,N} & j=i_0 \\
      -p_{j,N} & i=i_0
   \end{cases},
\] so (i) follows for $(\mathcal{P}'_N)_N$ as well. $(ii)$ and $(ii)'$ follow by the fact that $((\mathcal{P}')'_N)_N=(\mathcal{P}_N)_N.$

\medskip

$(7)$ Property $(i)$ is invariant under the vdC-operation.\footnote{ So, for sequences $(\mathcal{P}_N)_N$ with degree $\geq 2,$ the vdC-operation preserves the super niceness property.}

Indeed, if $p_{i_0,N}$ is the polynomial guaranteed by Lemma~\ref{L:type}, then we have the iterates:
$p_{1,N}(n+h)-p_{i_0,N}(n),\ldots, p_{\ell,N}(n+h)-p_{i_0,N}(n),$ and
$p_{1,N}(n)-p_{i_0,N}(n),\ldots,p_{i_0-1,N}(n)-p_{i_0,N}(n),p_{i_0+1,N}(n)-p_{i_0,N}(n),\ldots,p_{\ell,N}(n)-p_{i_0,N}(n).$

The degrees of these polynomials satisfy
\[\deg(p_{i,N}(n+h)-p_{i_0,N}(n))= \begin{cases}
      \deg(p_{i,N}(n)-p_{i_0,N}(n)) & i\neq i_0 \\
      \deg(p_{i_0,N})-1 & i=i_0
   \end{cases}.\footnote{ Notice that the vdC-operation will remove the $p_{i_0,N}(n+h)-p_{i_0,N}(n)$ iterate in case $\deg(p_{i_0,N})=1.$}\] 
   For the pairwise differences part, for $i\neq j,$ we have
   \[p_{i,N}(n+h)-p_{i_0,N}(n)-(p_{j,N}(n+h)-p_{i_0,N}(n))=p_{i,N}(n+h)-p_{j,N}(n+h),\]
   \[p_{i,N}(n)-p_{i_0,N}(n)-(p_{j,N}(n)-p_{i_0,N}(n))=p_{i,N}(n)-p_{j,N}(n),\] and, finally,
   \[p_{i,N}(n+h)-p_{i_0,N}(n)-(p_{j,N}(n)-p_{i_0,N}(n))=
      p_{i,N}(n+h)-p_{j,N}(n),\] so, everything follows by Property (i) for $(\mathcal{P}_N)_N.$\footnote{ Recall here that if, in the case where $i=j,$ it happens $\deg(p_{i,N})=1,$ then $\deg(p_{i_0,N})=1$ (as a non-constant polynomial of minimum degree in $(\mathcal{P}_N)_N$), so the vdC-operation will group the terms $p_{i,N}(n+h)-p_{i_0,N}(n)$ and $p_{i,N}(n)-p_{i_0,N}(n)$ together, being of degree $1$ with bounded difference.}
\end{remark}

Notice that Remark~\ref{R:super} (5) implies that Theorem~\ref{T:mc_new}, via Theorem~\ref{T:mc}, holds for a larger class of variable polynomial sequences; even with coefficients that oscillate.

A real-valued function $g$ which is continuously differentiable on $[c,\infty),$ where $c\geq 0,$ is called \emph{Fej\'{e}r} if the following hold:

$\bullet$ $g'(x)$ tends monotonically to $0$ as $x\to\infty;$ and

$\bullet$ $\lim_{x\to\infty}x|g'(x)|=\infty.$\footnote{ For a study of averages with general sublinear iterates one is referred to \cite{DKS}, and to \cite{BK} and \cite{K2} for more general functions, e.g. tempered functions.}

Any such function is eventually monotonic and satisfies the growth conditions $\log x \prec g(x) \prec x,$ hence $(1/g(N))_N$ has the $R_1$-property. So, modulo the goodness property, Theorem~\ref{T:mc} will also hold for polynomial sequences of the form:
\[\left(\frac{\sqrt{5}}{g_1(N)}n^3 +p_{1,N}(n)\right)_N, \;\text{or}\;\left(\frac{7}{g_2(N)}n^{17}+p_{2,N}(n)\right)_N,\] where $g_1(x)=x^{1/2}(2+\cos\sqrt{\log x}),$ $g_2(x)=x^{1/40}(1/10+\sin\log x)^3,$ and $p_{1,N},$ $p_{2,N}$ are polynomials of degrees less than $3$ and $17$ respectively with bounded coefficients. This is a non-trivial generalization because while the functions $g_1$ and $g_2$ are Fej\'{e}r, in view of the fact that they oscillate, do not belong to $\mathcal{SLE}.$


The following result shows that the nilfactor $\mathcal{Z}$ is characteristic for a super nice collection of polynomial sequences $(p_{1,N},\ldots,p_{\ell,N})_N.$

\begin{proposition}\label{L:characteristic_factor}
For $\ell\in \mathbb{N}$ let $(p_{1,N},\ldots,p_{\ell,N})_N$ be a super nice sequence of polynomials, $(X,\mathcal{B},\mu,T)$ a system, and suppose that at least one of the functions $f_1,\ldots,f_\ell\in L^\infty(\mu)$ is orthogonal to the nilfactor $\mathcal{Z}.$ Then, 
we have
\begin{equation}\label{E:characteristic_factor}
\lim_{N\to\infty}\norm{\frac{1}{N}\sum_{n=1}^N 
T^{[p_{1,N}(n)]}f_1\cdot\ldots\cdot T^{[p_{\ell,N}(n)]}f_\ell}_2=0.
\end{equation}
\end{proposition}

\begin{proof}
We assume without loss of generality that $f_1$ is orthogonal to $\mathcal{Z}.$ As in \cite[Lemma~4.7]{F2}, 
to show \eqref{E:characteristic_factor}, it suffices to show:
\begin{equation}\label{E:Char1}
\lim_{N\to\infty}\sup_{\norm{f_0}_\infty,\norm{f_2}_\infty,\ldots,\norm{f_\ell}_\infty\leq 1}\frac{1}{N}\sum_{n=1}^N \left| \int f_0\cdot T^{[p_{1,N}(n)]}f_1\cdot\ldots\cdot T^{[p_{\ell,N}(n)]}f_\ell\;d\mu \right|=0.
\end{equation}
We claim next that we can further assume that $\deg(p_{1,N})=\deg(\mathcal{P}_N).$ If this is not the case and $\deg(p_{1,N})<\deg(p_{i_0,N})=\deg(\mathcal{P}_N),$ then, precomposing with $T^{-[p_{i_0,N}(n)]},$ \eqref{E:Char1} becomes
\begin{equation*}\label{E:Char2}
	\begin{split}
	&\quad \lim_{N\to\infty}\sup_{\norm{f_0}_\infty,\norm{f_2}_\infty,\ldots,\norm{f_\ell}_\infty\leq 1}\frac{1}{N}\sum_{n=1}^N \bigg| \int f_{i_0}\cdot T^{[-p_{i_0,N}(n)]}f_{0,n,N}
	\\
&  \quad\quad\quad\quad\quad\quad\quad\quad\quad\quad\quad\quad\quad\quad\quad\quad\quad \cdot \prod_{1\leq i\neq i_0\leq \ell} T^{[p_{i,N}(n)-p_{i_0,N}(n)]}f_{i,n,N} \;d\mu\bigg|=0,
	\end{split}
    \end{equation*}
where $f_{i,n,N}=T^{e_i(n,N)}f_i$ for some $e_i(n,N)\in \{0,1\}.$ It suffices to show that for all $e\in \{0,1\}$
\begin{equation*}\label{E:Char3}
	\begin{split}
	& \lim_{N\to\infty}\sup_{\norm{f_0}_\infty,\norm{f_2}_\infty,\ldots,\norm{f_\ell}_\infty\leq 1}\frac{1}{N}\sum_{n=1}^N \bigg| \int f_{i_0}\cdot T^{[-p_{i_0,N}(n)]}f_0\cdot T^{[p_{1,N}(n)-p_{i_0,N}(n)]}(T^e f_{1})
	\\
&  \quad\quad\quad\quad\quad\quad\quad\quad\quad\quad\quad\quad\quad\quad\quad\quad \cdot \prod_{2\leq i\neq i_0\leq \ell} T^{[p_{i,N}(n)-p_{i_0,N}(n)]}f_{i} \;d\mu\bigg|=0.
	\end{split}
    \end{equation*}
The claim follows by Remark~\ref{R:super} (6), as the family $(p_{i_0,N}-p_{1,N},\ldots,p_{i_0,N}-p_{i_0-1,N},$ $p_{i_0,N},p_{i_0,N}-p_{i_0+1,N},\ldots,p_{i_0,N}-p_{\ell,N})_N$ is super nice with degree$=\deg(p_{i_0,N}-p_{1,N})$.

If all the $p_{i,N}$'s are of degree $1,$ the result follows from Proposition~\ref{P:Linear}. For $\deg(p_{1,N})\geq 2,$ we use induction on the type of the polynomial family of $\ell$-tuple of sequences.

For every $N\in \mathbb{N}$ we choose functions $f_{i,N}$ with $\norm{f_{i,N}}_\infty\leq 1$ for $i\in \{0,2,\ldots,\ell\},$ so that the average in \eqref{E:Char1} is $1/N$ close to  $\sup_{\norm{f_0}_\infty,\norm{f_2}_\infty,\ldots,\norm{f_\ell}_\infty\leq 1}.$ If $S=T\times T,$ $F_1=f_1\otimes\overline{f}_1,$ $F_{i,N}=f_{i,N}\otimes\overline{f}_{i,N},$ $ i=0,2,\ldots,\ell,$ and $\tilde{\mu}=\mu\times\mu,$  using Cauchy-Schwarz, \eqref{E:Char1} follows if
\begin{equation}\label{E:Char6}
\lim_{N\to\infty}\norm{\frac{1}{N}\sum_{n=1}^N  S^{[p_{1,N}(n)]}F_1\cdot\prod_{i=2}^\ell S^{[p_{i,N}(n)]}F_{i,N}}_{L^2(\tilde{\mu})}=0.
\end{equation}
By Lemma~\ref{L:vdc}, \eqref{E:Char6} follows if, for large enough $h,$ for $N\to\infty$ we have
\begin{equation}\label{E:SR}
	\frac{1}{N}\sum_{n=1}^N \bigg| \int S^{[p_{1,N}(n+h)]}F_1\cdot\prod_{i=2}^\ell S^{[p_{i,N}(n+h)]}F_{i,N} \cdot S^{[p_{1,N}(n)]}\overline{F}_1\cdot\prod_{i=2}^\ell S^{[p_{i,N}(n)]}\overline{F}_{i,N} \;d\tilde{\mu}\bigg|
    \end{equation} goes to $0.$
Picking $p_{j_0,N}$ as guaranteed by Lemma~\ref{L:type} (the degrees of the $p_{i,N}$'s are fixed, so the choice of $j_0$ is independent of $N$), we precompose with the term $S^{-[p_{j_0,N}(n)]}$ in the integrand of \eqref{E:SR}
(notice that some error terms $e_i(n,h,N),$ $\tilde{e}_i(n,N)\in \{0,1\}$ will appear). Next, we group the degree $1$ iterates. More specifically, if $\deg(p_{i,N})=1$ for some $2\leq i\leq \ell$, then  $[p_{i,N}(n+h)]=[p_{i,N}(n)]+[c_{i,N}h]+e'_{i}(n,h,N),$ for some error terms in $\{0,1\}.$ Hence, we have
\begin{equation*}
	\begin{split}
	&\quad S^{[p_{i,N}(n+h)-p_{j_0,N}(n)]+e_i(n,h,N)}F_{i,N}\cdot S^{[p_{i,N}(n)-p_{j_0,N}(n)]+\tilde{e}_i(n,h,N)}\overline{F}_{i,N}
\\& =S^{[p_{i,N}(n)-p_{j_0,N}(n)]}(S^{[c_{i,N}h]+e_i(n,h,N)+e'_i(n,h,N)}F_{i,N}\cdot S^{\tilde{e}_{i}(n,h,N)}\overline{F}_{i,N}).
	\end{split}
    \end{equation*}
We treat this product as one iterate. After this grouping, assuming that $r$ many terms remain, it suffices to show, for large $h,$ and every choice of $e\in \{0,1\}$ that
\begin{equation}\label{E:ind2}
	\begin{split}
	&\quad\lim_{N\to\infty}\sup_{\norm{F_0}_\infty,\norm{F_2}_\infty,\ldots,\norm{F_{r}}_\infty\leq 1}\bigg|\frac{1}{N}\sum_{n=1}^N \int F_0\cdot S^{[p_{1,h,N}(n)]}(S^{e}F_1)
	\\
	& \quad\quad\quad\quad\quad\quad\quad\quad\quad\quad\quad \quad\quad\quad\quad\quad\quad \quad\quad\quad\cdot \prod_{i=2}^{r}S^{[p_{i,h,N}(n)]}F_i\;d\tilde{\mu}\bigg|=0,
	\end{split}
    \end{equation}
where the polynomial sequences $(p_{i,h,N})_N$ form  $(\mathcal{P}_N(p_{j_0,N},h))_N,$ a polynomial family with $p_{1,h,N}(n)=p_{1,N}(n+h)-p_{j_0,N}(n)$ and $\deg(p_{1,h,N})=\deg(\mathcal{P}_N(p_{j_0,N},h)).$

The
left-hand side of \eqref{E:ind2} is as that of \eqref{E:Char1}, with the polynomial family of the former having type strictly less than the latter (from Lemma~\ref{L:type}).  Using Remark~\ref{R:super} (7), we are done by induction.
\end{proof}

For the expression of Theorem~\ref{T:mc}, writing $i[p_N(n)]=[ip_N(n)]+e_{i,N}(n),$ for some  $e_{i,N}(n)\in \{-i,\ldots,-1,0\},$ using Proposition~\ref{L:characteristic_factor}, we get the following result:

\begin{proposition}\label{L:characteristic_single}
For $\ell\in \mathbb{N}$ let $(p_{N},2 p_N,\ldots,\ell p_{N})_N$ be a super nice sequence of polynomials, $(X,\mathcal{B},\mu,T)$ a system, and suppose that at least one of the functions $f_1,\ldots,f_\ell\in L^\infty(\mu)$ is orthogonal to the nilfactor $\mathcal{Z}.$ Then, 
we have
\begin{equation*}\label{E:characteristic_single}
\lim_{N\to\infty}\norm{\frac{1}{N}\sum_{n=1}^N 
T^{[p_{N}(n)]}f_1\cdot T^{2[p_{N}(n)]}f_2\cdot\ldots\cdot T^{\ell[p_{N}(n)]}f_\ell}_2=0.
\end{equation*}
\end{proposition}

\section{Equidistribution}\label{S6:equi}

In order to prove our main equidistribution result (Theorem~\ref{T:T1}), we start with some definitions and facts, following  \cite{F1} (see \cite[Subsubsection~2.3.2]{F1} for more details).

If $G$ is a nilpotent group, then a sequence
$g:\mathbb{N}\to G$ of the form $g(n) = b_1^{p_1(n)} \cdots b_k^{p_k(n)},$ where $b_i \in G,$ and $p_i$ are integer polynomials, is called a \emph{polynomial sequence in $G$}. If the maximum degree
of the polynomials $p_i$'s is at most $d$ we say that the \emph{degree} of $g(n)$ is at most $d.$

Given a nilmanifold $X = G/\Gamma$ the \emph{horizontal torus} is defined to be the compact abelian group $Z = G/([G,G]\Gamma)$. If $X$ is connected, then $Z$ is isomorphic to some finite dimensional torus $ \mathbb{T}^s $.
A \emph{horizontal character}
$ \chi: G \to \mathbb{C}$ is a continuous homomorphism that satisfies $ \chi(g\gamma) = \chi(g) $ for every $ \gamma \in \Gamma $
and can be thought of as a character of $\mathbb{T}^s$, in which case there exists a unique $\kappa \in \mathbb{Z}^s$ such that $\chi(t\mathbb{Z}^s) = e(\kappa \cdot t)$, where ``$\cdot$'' denotes the inner product operation, and $e(x):=e^{2\pi i x}$.



Let $p: \mathbb{Z} \to \mathbb{R}$ be a polynomial sequence of degree $d$ of the form $p(n)=\sum_{i=0}^d a_i n^i$, where $a_i\in \mathbb{R},$ $1\leq i\leq d.$ We define the \emph{smoothness norm} by
\begin{equation}\label{E:norm}
    \norm{e(p(n))}_{C^\infty[N]}:= \max_{1\leq i \leq d} (N^i \norm{a_i}),
\end{equation}
where $\norm{\cdot}$ denotes the distance to the closest integer, i.e., $\norm{x}:=d(x,\mathbb{Z})$.

Given $ N\in \mathbb{N} $, a finite sequence $ (g(n)\Gamma)_{1\leq n \leq N} $ is said to be  $ \delta $-\emph{equidistributed in $X$}, if
\[  \left| \frac{1}{N}\sum_{n=1}^N F(g(n)\Gamma) -\int_X F d m_X \right| \leq \delta \norm{F}_{\lip(X)}    \]
for every Lipschitz function $F:X\to \mathbb{C},$ where
\[ \norm{F}_{\lip(X)}= \norm{F}_\infty + \sup_{x,y\in X, x\neq y} \frac{|F(x)-F(y)|}{d_X(x,y)} \]
for some appropriate metric $d_X$.





At this point we quote \cite[Theorem~2.9]{F1}, a direct consequence of \cite[Theorem~2.9]{GT}:

\begin{theorem}[Green \& Tao, \cite{GT}] \label{T:t1}
Let $X=G/\Gamma$ be a nilmanifold with $G$ connected and simply connected, and $d\in \mathbb{N}$. Then for every small enough $\delta >0$ there exist a positive constant $M \equiv M(X,d,\delta)$ with the following property: For every $N \in \mathbb{N}$, if $g : \mathbb{Z} \to G$ is a polynomial sequence of degree at most $d$ such that the finite
sequence $(g(n)\Gamma)_{1\leq n\leq N}$ is not $\delta$-equidistributed, then for some non-trivial horizontal character $\chi$ with $ \norm{\chi} \leq M$ we have
\begin{equation*}\label{E:e31}
    \norm{\chi(g(n))}_{C^\infty[N]}\leq M
\end{equation*}
($\chi$ here is thought of as a character of the horizontal torus $Z = \mathbb{T}^s$ and $g(n)$ as a polynomial sequence in $\mathbb{T}^s$).
\end{theorem}




Adapting the notion of equidistribution of a sequence in a nilmanifold (recall \eqref{E:equi}) to our case, abusing the notation, we say that $(b^{a_N(n)}x)_{1\leq n\leq N},$ where $(a_N(n))_{1\leq n\leq N}$ is a variable sequence of real numbers and $X=G/\Gamma$ is a nilmanifold with $G$ connected and simply connected,
is \emph{equidistributed} in a subnilmanifold $Y$ of $X,$ if for every $F\in C(X)$ we have
\[\lim_{N\to\infty}\frac{1}{N}\sum_{n=1}^N F(b^{a_N(n)}x)=\int F\; d m_Y.\]




 In order for us to prove Theorems~\ref{T:main} and \ref{T:mc}, we prove the following equidistribution theorem, which is the main result of this section:

\begin{theorem}\label{T:T1}
Let $(p_{1,N}, \ldots, p_{\ell,N})_N$ be a good sequence of $\ell$-tuples of polynomials.
\begin{enumerate}
\item[$(i)$] If $X_i=G_i/\Gamma_i,$ $1\leq i\leq \ell,$ are nilmanifolds with $G_i$ connected and simply connected, then for every $b_i\in G_i$ and $x_i\in X_i$ the sequence \[(b_1^{p_{1,N}(n)}x_1,\ldots, b_\ell^{p_{\ell,N}(n)}x_\ell)_{1\leq n\leq N}\] is equidistributed in the nilmanifold $\overline{(b_1^sx_1)}_{s\in\mathbb{R}}\times\cdots\times \overline{(b_\ell^sx_\ell)}_{s\in\mathbb{R}}.$
\item[$(ii)$] If $X_i=G_i/\Gamma_i,$ $1\leq i\leq \ell,$ are nilmanifolds, then for every $b_i\in G_i$ and $x_i\in X_i$ the sequence \[(b_1^{[p_{1,N}(n)]}x_1,\ldots, b_\ell^{[p_{\ell,N}(n)]}x_\ell)_{1\leq n\leq N}\] is equidistributed in the nilmanifold $\overline{(b_1^nx_1)}_{n}\times\cdots\times \overline{(b_\ell^nx_\ell)}_{n}.$
\end{enumerate}
\end{theorem}

\begin{remark}\label{R:equidtr}
In order to prove Theorem~\ref{T:T1}, we can assume that $X_1=\ldots=X_\ell=X.$

Indeed, in the general case we consider the nilmanifold $\tilde{X}=X_1\times\cdots\times X_\ell.$ Then $\tilde{X}=\tilde{G}/\tilde{\Gamma},$ where $\tilde{G}=G_1\times \cdots\times G_\ell$ is connected and simply connected and $\tilde{\Gamma}=\Gamma_1\times\cdots\times\Gamma_\ell$ is a discrete cocompact subgroup of $\tilde{G}.$ Each $b_i$ can be considered as an element of $\tilde{G}$ and each $x_i$ as an element of $\tilde{X}.$ Changing the base point we can also assume that $x=\Gamma.$
\end{remark}


Part (ii) of the previous result follows from Part (i) (see \cite[Lemma~5.1]{F1}):


\begin{lemma}\label{L:5.1}
Let $\ell\in\mathbb{N}$ and $(a_{1,N},\ldots,a_{\ell,N})_{N}$ be sequence of $\ell$-tuples of real numbers. Suppose that for every nilmanifold $X=G/\Gamma,$ with $G$ connected and simply connected, and every $b_1,\ldots,b_\ell\in G$ the sequence \[(b_1^{a_{1,N}(n)}\Gamma,\ldots,b_\ell^{a_{\ell,N}(n)}\Gamma)_{1\leq n\leq N}\] is equidistributed in the nilmanifold  $\overline{(b_1^s\Gamma)}_{s\in\mathbb{R}}\times\cdots\times \overline{(b_\ell^s\Gamma)}_{s\in\mathbb{R}}.$ Then, for every nilmanifold $X=G/\Gamma,$ $b_1,\ldots,b_\ell\in G$ and $x_1,\ldots,x_\ell\in X,$ the sequence \[(b_1^{[a_{1,N}(n)]}x_1,\ldots, b_\ell^{[a_{\ell,N}(n)]}x_\ell)_{1\leq n\leq N}\] is equidistributed in the nilmanifold $\overline{(b_1^nx_1)}_{n}\times\cdots\times \overline{(b_\ell^nx_\ell)}_{n}.$
\end{lemma}

\begin{proof}[Sketch of the proof]
Following \cite[Lemma~4.1]{F1}, we show the $\ell=1$ case, as the general one follows with some straightforward modifications.

Let $X=G/\Gamma$ be a nilmanifold, $b\in G$ and $x\in X.$ Using some standard reductions (namely, the lifting argument and the change of base point formula from Subsections~\ref{Ss:3.2.3} and ~\ref{Ss:3.2.2}), we can and will assume that $G$ is connected and simply connected and that $x=\Gamma.$

Letting $X_b:=\overline{(b^n\Gamma)}_{n}$ and $m_{X_b}$ the corresponding normalized Haar measure, we will show that for every $F\in C(X)$ we have
\begin{equation}\label{E:4.1_1}
\lim_{N\to\infty}\frac{1}{N}\sum_{n=1}^N F(b^{[a_N(n)]}\Gamma)=\int_{X_b}F\ dm_{X_b}.
\end{equation}
Using our assumption for the case $\tilde{X}:=\tilde{G}/\tilde{\Gamma},$ where $\tilde{G}:=\mathbb{R}\times G$ is connected and simply connected, $\tilde{\Gamma}:=\mathbb{Z}\times\Gamma$ and $\tilde{b}:=(1,b),$ for every $H\in C(\tilde{X})$ we have
\begin{equation}\label{E:4.1_2}
\lim_{N\to\infty}\frac{1}{N}\sum_{n=1}^N H(\tilde{b}^{a_N(n)}\tilde{\Gamma})=\int_{\tilde{X}_{\tilde{b}}} H\ d m_{\tilde{X}_{\tilde{b}}},
\end{equation}
where $\tilde{X}_{\tilde{b}}:=\overline{(s\mathbb{Z},b^s\Gamma)}_{s\in \mathbb{R}}$ and $m_{\tilde{X}_{\tilde{b}}}$ is its corresponding normalized Haar measure.\footnote{ Here we adapt the notation $z\mathbb{Z}$ which is more convenient than $z (\text{mod} 1).$}

Let $F\in C(X),$ and define $\tilde{F}:\tilde{X}\to \mathbb{C}$ with $\tilde{F}(t\mathbb{Z},g\Gamma):=F(b^{-\{t\}}g\Gamma).$ While $\tilde{F}$ may be discontinuous, for every $0<\delta<1/2$ there exists $\tilde{F}_\delta\in C(\tilde{X})$ that equals $\tilde{F}$ on $\tilde{X}_\delta=I_\delta\times X,$ where $I_\delta=\{t\mathbb{Z}:\;\norm{t}\geq \delta\},$ and it is uniformly bounded by $2\norm{F}_\infty.$

Since $\tilde{b}^{a_N(n)}=(a_N(n),b^{a_N(n)}),$
our assumption implies that $a_N(n)\mathbb{Z}\in I_\delta,$ and so $\tilde{b}^{a_N(n)}\tilde{\Gamma}\in \tilde{X}_\delta,$ for a set of $n$'s with density $1-2\delta.$\footnote{ By this we mean $\lim_{N\to\infty}N^{-1}\cdot|\{1\leq n\leq N:\;a_N(n)\mathbb{Z}\in I_\delta\}|=1-2\delta.$}
 So,
\[\limsup_{N\to\infty}\frac{1}{N}\sum_{n=1}^N |\tilde{F}(\tilde{b}^{a_N(n)}\tilde{\Gamma})-\tilde{F}_\delta(\tilde{b}^{a_N(n)}\tilde{\Gamma})|\leq 4\delta \norm{F}_{\infty},\] hence, since \eqref{E:4.1_2} holds for every $\tilde{F}_\delta,$ it also holds for $\tilde{F}.$

The map $(s\mathbb{Z},g\Gamma)\mapsto b^{-\{s\}}g\Gamma$ sends $\tilde{X}_{\tilde{b}}$ onto $X_b.$ Defining the measure $m$ on $X_b$ by
\[\int_{X_b} F\ dm:= \int_{\tilde{X}_{\tilde{b}}}F(b^{-\{s\}}g\Gamma)\ dm_{\tilde{X}_{\tilde{b}}}(s\mathbb{Z},g\Gamma),\] we have (see \cite[Lemma~4.1]{F1} for details) that $m=m_{X_b}.$ Thus,
since
\[\tilde{F}(\tilde{b}^{a_N(n)}\tilde{\Gamma})=\tilde{F}(a_N(n)\mathbb{Z},b^{a_N(n)}\Gamma)=F(b^{-\{a_N(n)\}}b^{a_N(n)}\Gamma)=F(b^{[a_N(n)]}\Gamma),\]
using \eqref{E:4.1_2} for the function $\tilde{F},$ we get
\begin{equation*}
\begin{split}
&\lim_{N\to\infty}\frac{1}{N}\sum_{n=1}^N F(b^{[a_N(n)]}\Gamma)=\int_{\tilde{X}_{\tilde{b}}}\tilde{F}\ d m_{\tilde{X}_{\tilde{b}}}\\
=&\int_{\tilde{X}_{\tilde{b}}}F(b^{-\{s\}}g\Gamma)\ dm_{\tilde{X}_{\tilde{b}}}(s\mathbb{Z},g\Gamma)=\int_{X_b} F\ d m_{X_b},
\end{split}
\end{equation*}
so we have \eqref{E:4.1_1}.
\end{proof}

Recalling that a sequence of $\ell$-tuples of variable polynomials $(p_{1,N}, \ldots, p_{\ell,N})_N$ is good if every non-trivial linear combination of $(p_{1,N})_N, \ldots, (p_{\ell,N})_N$ is good, we have the following:

\begin{lemma}\label{L:equi_multi}
Let $(p_{1,N}, \ldots, p_{\ell,N})_N$ be a good sequence of $\ell$-tuples of polynomials, $X_i=G_i/\Gamma_i$ nilmanifolds, with $G_i$ connected and simply connected, and suppose that $b_i\in G_i$ acts ergodically on $X_i,$ $1\leq i\leq \ell$. Then the sequence \[(b_1^{p_{1,N}(n)}\Gamma_1,\ldots,b_\ell^{p_{\ell,N}(n)}\Gamma_\ell)_{1\leq n\leq N}\] is equidistributed in $X_1\times\cdots\times X_\ell$.
\end{lemma}

\begin{proof}
We follow \cite[Lemma~5.3]{F1}.
As the general case is similar, we assume that $X_1=\ldots=X_\ell=X.$  Arguing by contradiction, we will also assume that for some $\delta>0,$ $(b_1^{p_{1,N}(n)}\Gamma,\ldots,b_\ell^{p_{\ell,N}(n)}\Gamma)_{1\leq n\leq N}$ is not $\delta$-equidistributed in $X^\ell.$

 If $p_{i,N}(t)=\sum_{k=0}^{d_i} c_{i,k,N}t^k,$ then
\[b_i^{p_{i,N}(n)}=b_{i,0,N}\cdot b_{i,1,N}^{n}\cdots b_{i,d_i,N}^{n^{d_i}},\] where $b_{i,j,N}=b^{c_{i,j,N}},$ $0\leq j\leq d_i,$ $1\leq i\leq \ell,$ so, for all $N\in \mathbb{N},$ $(b_1^{p_{1,N}(n)},\ldots,b_\ell^{p_{\ell,N}(n)})_n$ is a polynomial sequence in $G^\ell.$

Applying  Theorem~\ref{T:t1}, we have a constant $M\equiv M(\delta, X,d_1,\ldots,d_\ell)$ and a horizontal character $\chi$ of $X^\ell$ with $\norm{\chi}\leq M$ such that
\[\norm{\chi(b_1^{p_{1,N}(n)},\ldots,b_\ell^{p_{\ell,N}(n)})}_{C^\infty[N]}\leq M.\]
Let $\pi(b_i)=(\beta_{i,1}\mathbb{Z},\ldots,\beta_{i,s}\mathbb{Z}),$ $1\leq i\leq \ell,$ where $\beta_{i,j}\in \mathbb{R},$ be the projection of $b_i$ on the horizontal torus $\mathbb{T}^s$ (the integer $s$ is bounded by the dimension of $X$).  Using the ergodicity assumption on the $b_i$'s, for all $1\leq i\leq \ell,$ the set $\{1,\beta_{i,1},\ldots,\beta_{i,s}\}$ consists of rationally independent elements. For $t\in \mathbb{R}$ we have $\pi(b_i^t)=(t\tilde{\beta}_{i,1}\mathbb{Z},\ldots,t\tilde{\beta}_{i,s}\mathbb{Z})$ for some elements $\tilde{\beta}_{i,j}\in \mathbb{R}$ with $\tilde{\beta}_{i,j}\mathbb{Z}=\beta_{i,j}\mathbb{Z},$\footnote{ Note here that, for all $1\leq i\leq \ell,$ the $\tilde{\beta}_{i,j}$'s are also rationally independent.}
so, we have that
\[\chi(b_1^{p_{1,N}(n)},\ldots,b_\ell^{p_{\ell,N}(n)})=e\left(\sum_{i=1}^\ell p_{i,N}(n)\sum_{j=1}^s\lambda_{i,j} \tilde{\beta}_{i,j}\right)\] for some integers $\lambda_{i,j}\in \mathbb{Z}.$

If, for $n\in \mathbb{N},$ we set
\[p_N(n):=\sum_{i=1}^\ell p_{i,N}(n)\sum_{j=1}^s\lambda_{i,j} \tilde{\beta}_{i,j}=\sum_{k=0}^d c_{k,N}n^k,\]  we have that the sequence $(p_N)_N,$ being a non-trivial (as $\chi$ is non-trivial and $\tilde{\beta}_{i,j}$'s are rationally independent) linear combination of the $(p_{i,N})_N$'s, is good. Combining the last three relations, we get
$M\geq \norm{e(p_N(n))}_{C^\infty[N]}\geq \max_{1\leq j\leq d}\left(N^j\norm{c_{j,N}}\right),$ which is a contradiction to $\lim_{N\to\infty}\max_{1\leq j\leq d}\left(N^j\norm{c_{j,N}}\right)=\infty$; a condition that the coefficients of a good variable polynomial sequence satisfy (see \cite{F4}).   
\end{proof}

The last ingredient in proving Part  (i)  of Theorem~\ref{T:T1} is the following lemma:

\begin{lemma}[Lemma~5.2, \cite{F1}]\label{L:5.2}
Let $X=G/\Gamma$ be a nilmanifold with $G$ connected and simply connected. Then, for every $b_1,\ldots,b_\ell\in G,$ there exists an $s_0\in \mathbb{R}$ such that for all $1\leq i\leq \ell$ the element $b_i^{s_0}$ acts ergodically on the nilmanifold $\overline{(b_i^s \Gamma)}_{s\in\mathbb{R}}.$
\end{lemma}

We are now ready to prove Theorem \ref{T:T1}.

\begin{proof}[Proof of Theorem~\ref{T:T1}]
Using Lemma \ref{L:5.1} we see that Part (ii) of Theorem \ref{T:T1} follows from Part
(i). To establish Part (i) let $ b_1,\ldots, b_\ell \in G $. By Lemma~\ref{L:5.2} there exists a non-zero $ s_0\in \mathbb{R} $ such that for every $1\leq i\leq \ell$ the element $b_i^{s_0}$ acts ergodically on the nilmanifold $\overline{(b_i^s \Gamma)}_{s\in\mathbb{R}}.$ Using Lemma~\ref{L:equi_multi} for the elements $ b_i^{s_0} $ and the polynomials $ p_{i,N}/s_0 $ (which are still forming a good sequence of $\ell$-tuples of polynomials) we get that the sequence
$(b_1^{p_{1,N}(n)}\Gamma,\ldots,b_\ell^{p_{\ell,N}(n)}\Gamma)_{1\leq n\leq N}$ is equidistributed in the nilmanifold  $\overline{(b_1^s\Gamma)}_{s\in\mathbb{R}}\times\cdots\times \overline{(b_\ell^s\Gamma)}_{s\in\mathbb{R}}$, hence we get the conclusion.
\end{proof}

\section{Proof of main results}

To prove our main results, we first show that the polynomial sequences from Theorems~\ref{T:main_new} and \ref{T:mc_new} are good and super nice. If either $g_2\prec g_1$ or $g_2\sim g_1,$ we write $g_2\precsim g_1.$


\begin{lemma}\label{L:are_good}
The polynomial sequences from Theorems~\ref{T:main_new} and \ref{T:mc_new} are good.
\end{lemma}

\begin{proof}
Let $\lambda_1 p_{1,N}+\ldots+\lambda_\ell p_{\ell,N}$ be a non-trivial linear combination of strongly independent variable polynomials as in \eqref{E:polyv}, which is also of the same form. In case this combination is a polynomial of degree $1,$ precomposing with the opposite of its constant term, without loss of generality, we can assume that it is of the form $h(N)n,$ where $h(N)\sim 1/g(N),$ with $1\prec g(N)\prec N$ (hence $h(N)\to 0$ and $|h(N)|N\to\infty$ monotonically as $N\to\infty$). For any $\alpha\neq 0,$ as $N\to\infty,$ we have that
\[\frac{1}{N}\sum_{n=0}^{N-1}e^{i\alpha h(N)n}=\frac{1}{N}\cdot\frac{1-e^{i\alpha h(N)N}}{1-e^{i\alpha h(N)}}=\frac{h(N)}{1-e^{i\alpha h(N)}}\cdot\frac{1-e^{i\alpha h(N)N}}{h(N)N}\to \frac{i}{\alpha}\cdot0=0.\]
In case the combination is a polynomial of degree $d,$ after using Lemma~\ref{L:vdc} $(d-1)$ times, we get a polynomial of degree $1,$ hence the result follows from the previous step.
\end{proof}

Recall that when we want to check that a $k$-tuple, for $k>1,$ has the $R_k$-property, we have to check (according to Definition~\ref{D:R_property}) that for every $1\leq i\leq k$ the corresponding $(k-1)$-tuple has the $R_{k-1}$-property. If a $(k-1)$-tuple corresponds to the index $i_0,$ we say that it is \emph{descending} from the $i_0$ term of the previous step.

\begin{lemma}\label{L:are_super_nice}
The polynomial sequences from Theorems~\ref{T:main_new} and \ref{T:mc_new} are super nice.
\end{lemma}

\begin{proof}
For a single polynomial sequence as in \eqref{E:polyv}, the result follows immediately from Remark~\ref{R:super} (5) and the properties of Hardy field functions.

For multiple sequences, $(i)$ follows by the form \eqref{E:polyv} that the variable polynomial sequences have. As $(\mathcal{P}_N)_N$ and $(\mathcal{P}'_N)_N$ consist of polynomials of the same form, $(ii)$ and $(ii)'$ will both follow by the same argument.

After performing, if needed, finitely many vdC-operations to the polynomial families of interest, assuming that we have $k$ many essentially distinct terms of the form $a_{i,N}n,$ $1\leq i\leq k,$ we have $a_{i,\cdot}\in \mathcal{C}(g_1,\ldots,g_l)$ and $a_{k,N}\precsim\ldots\precsim a_{2,N}\precsim a_{1,N}.$\footnote{ This happens because vdC-operations preserve the essential distinctness property of the polynomials and at each step the coefficient functions belong to $\mathcal{C}(g_1,\ldots,g_l).$} In order to show that the sequences $\{(a_{i,N})_N:\;1\leq i\leq k\}$ have the $R_k$-property, we present an algorithmic way of finding the corresponding terms at the steps $1$ and $\lambda,$ for $\lambda\geq 2$:

\medskip

\noindent {\bf{Step 1:}} For $i=1$ we pick $j_0=k$ (i.e., the largest index). In this case we will show that we have property (ii) (a) (of Definition~\ref{D:R_property}). The terms become:
\[\frac{a_{k,N}-a_{j,N}}{a_{1,N}}\sim \frac{a_{j,N}}{a_{1,N}}, \;\;1\leq j\leq k-1.\]

For $i>1,$ we pick $j_0=1$ (i.e., the smallest index). In this case we will show that we have property (ii) (b). The terms become:
\[\frac{a_{j,N}}{a_{i,N}-a_{1,N}}\sim \frac{a_{j,N}}{a_{1,N}}, \;\;2\leq j\leq k.\]
\noindent {\bf{Step $\lambda$:}} After we order them from largest to smallest growth, we denote the $j$-th term at the $\lambda$-th step with $a_{\lambda,j,N}$. We have two cases:

\medskip

$\bullet$ The sequence of coefficients is descending from the $i=1$ term of the $(\lambda-1)$-th step.

For $i=1$ we pick $j_0=k-\lambda+1$ and show property (ii) (a) (for the $i=1$ case we always pick the largest index $j_0$ and show property (ii) (a)). For $1\leq j\leq j_0-1,$ we have
\[\frac{a_{\lambda,j_0,N}-a_{\lambda,j,N}}{a_{\lambda,1,N}}=\frac{a_{\lambda-1,j,N}-a_{\lambda-1,j_0,N}}{a_{\lambda-1,j_0+1,N}-a_{\lambda-1,1,N}}\sim \frac{a_{\lambda-1,j,N}}{a_{\lambda-1,1,N}},\]
where the numerator comes from the difference $(a_{\lambda-1,j_0+1,N}-a_{\lambda-1,j_0,N})-\break(a_{\lambda-1,j_0+1,N}-a_{\lambda-1,j,N}),$ and the (common) denominators are canceled.

For $i>1$ we pick $j_0=1$ and show property (ii) (b) (for the $i>1$ case we always pick $j_0=1$ and show property (ii) (b)). For $2\leq j\leq k-\lambda+1$ we have
\[\frac{a_{\lambda,j,N}}{a_{\lambda,i,N}-a_{\lambda,1,N}}=\frac{a_{\lambda-1,k-\lambda+2,N}-a_{\lambda-1,j,N}}{a_{\lambda-1,1,N}-a_{\lambda-1,i,N}}\sim \frac{a_{\lambda-1,j,N}}{a_{\lambda-1,1,N}},\] where the denominator comes from the difference $(a_{\lambda-1,k-\lambda+2,N}-a_{\lambda-1,i,N})-(a_{\lambda-1,k-\lambda+2,N}$ $-a_{\lambda-1,1,N}),$ and, as in the previous case, the (common) denominators are canceled.

\medskip

$\bullet$ The sequence of coefficients is descending from the $i>1$ term of the $(\lambda-1)$-th step.

 For $i=1,$ we choose $j_0=k-\lambda+1$. For all $1\leq j\leq j_0-1$ we have:
\[\frac{a_{\lambda,j_0,N}-a_{\lambda,j,N}}{a_{\lambda,1,N}}=\frac{a_{\lambda-1,j_0+1,N}-a_{\lambda-1,j+1,N}}{a_{\lambda-1,2,N}}\sim\frac{a_{\lambda-1,j+1,N}}{a_{\lambda-1,2,N}}.\]

For $i>1,$ we choose $j_0=1$. For all $2\leq j\leq k-\lambda+1,$ we have
\[\frac{a_{\lambda,j,N}}{a_{\lambda,i,N}-a_{\lambda,1,N}}=\frac{a_{\lambda-1,j+1,N}}{a_{\lambda-1,i+1,N}-a_{\lambda-1,2,N}}\sim \frac{a_{\lambda-1,j+1,N}}{a_{\lambda-1,2,N}}.\]
Note that each of the aforementioned terms, at each step, is (up to a sign) of the form
\[\frac{a_{i,N}-a_{j,N}}{a_{s,N}-a_{t,N}},\;s>i>j\geq t,\; \frac{a_{i,N}-a_{j,N}}{a_{t,N}},\;i>j\geq t,\;\text{or},\;\frac{a_{j,N}}{a_{i,N}-a_{t,N}},\;t<\min\{i,j\},\] i.e., combinations of terms from the initial sequence (because of the cancellations mentioned above) which are all $\sim$ to $a_{j,N}/a_{t,N}.$ The claim now follows by the properties of elements from $\mathcal{LE}$ as each coefficient is a logarithmico-exponential Hardy function, hence eventually monotone, which is either $\sim 1$ or $\sim 1/g(N)$ with $1\prec g(N)\prec N$ by the construction.
\end{proof}

We now prove Theorem~\ref{T:main} (which implies Theorem~\ref{T:main_new}):

\begin{proof}[Proof of Theorem~\ref{T:main}]
We start by using Proposition~\ref{L:characteristic_factor} in order to get that the nilfactor
$ \mathcal{Z} $ is characteristic for the multiple average in \eqref{E:main33} (which can be used as the polynomial iterates are super nice). Via Theorem~\ref{T:HK} we can assume without loss of generality that our system is an inverse limit of nilsystems. By a standard approximation argument, we can further assume that it is actually a nilsystem.

Let $(X=G/\Gamma,\mathcal{G}/\Gamma,m_X,T_b)$ be a nilsystem, where $ b\in G $ is ergodic, and $ F_1,\ldots,F_\ell\in L^\infty(m_X) $. Our objective now is to show that 
\begin{equation}\label{E:e3}
\lim_{N\to \infty}\frac{1}{N}\sum_{n=1}^N F_1(b^{[p_{1,N}(n)]}x)\cdot \ldots \cdot F_\ell(b^{[p_{\ell,N}(n)]}x)= \int F_1 \; dm_X \cdot \ldots \cdot \int F_\ell \; dm_X,
\end{equation}
where the convergence takes place in $ L^2(m_X) $. By density, we can assume that the functions $ F_1,\ldots,F_\ell $ are continuous. In this case we will show that \eqref{E:e3} holds for all $x\in X,$ hence we will obtain the result by using the Dominated Convergence Theorem. By applying Theorem~\ref{T:T1} to the nilmanifold $ X^\ell $, the nilrotation
$ \tilde{b}=(b,\ldots, b)\in G^\ell $, the point
$ \tilde{x}=(x,\ldots, x)\in X^\ell $, $x\in X,$ and the continuous function
$ \tilde{F}(x_1,\ldots,x_\ell)=F_1(x_1)\cdot\ldots\cdot F_\ell(x_\ell) $ (here we are using the goodness property of the $p_{i,N}$'s), we get that
\[ \lim_{N\to \infty}\frac{1}{N}\sum_{n=1}^N \tilde{F}(b^{[p_{1,N}(n)]}x, \ldots, b^{[p_{\ell,N}(n)]}x)= \int \tilde{F} \; dm_{X^\ell}.\]
This implies that \eqref{E:e3} holds for every $x\in X$, completing the proof.
\end{proof}

Next, we show Theorem~\ref{T:mc} (which in turn implies Theorem~\ref{T:mc_new}):

\begin{proof}[Proof of Theorem~\ref{T:mc}]
 As in the previous proof, our objective is to show that if $(p_N)_N\subseteq \mathbb{R}[t]$ is a good polynomial sequence with $(p_N,2 p_N,\ldots,\ell p_N)_N$ being super nice, then for every nilsystem $(X=G/\Gamma,\mathcal{G}/\Gamma,m_X,T_b)$, where $ b\in G $ is ergodic, and $ F_1,\ldots,F_\ell\in L^\infty(m_X),$ we have that the limit
\begin{equation}\label{E:e1}
\lim_{N\to \infty}\frac{1}{N}\sum_{n=1}^N F_1(b^{[p_N(n)]}x)\cdot F_2(b^{2[p_N(n)]}x)\cdot \ldots \cdot F_\ell(b^{\ell[p_N(n)]}x)
\end{equation}
is equal to the limit
\begin{equation}\label{E:e2}
\lim_{N\to \infty}\frac{1}{N}\sum_{n=1}^N F_1(b^nx)\cdot F_2(b^{2n}x)\cdot \ldots \cdot F_\ell(b^{\ell n}x).
\end{equation}
As in the proof above, we assume that every $F_i$ is continuous. Applying Theorem~\ref{T:T1} to $X^\ell,$ the nilrotation $\tilde{b}=(b,b^2,\ldots,b^\ell),$ the point $\tilde{x}=(x,x,\ldots,x),$ $x\in X,$ and the continuous function $\tilde{F}(x_1,\ldots,x_\ell)=F_1(x_1)\cdot F_2(x_2)\cdot\ldots\cdot F_\ell(x_\ell),$  we get
\[ \lim_{N\to\infty}\frac{1}{N}\sum_{n=1}^N\tilde{F}(\tilde{b}^{[p_N(n)]}\tilde{x})=\lim_{N\to\infty}\frac{1}{N}\sum_{n=1}^N\tilde{F}(\tilde{b}^{n}\tilde{x}).\]
This implies that the limits in \eqref{E:e1} and \eqref{E:e2} exist for every $ x\in X $ and  are equal.
\end{proof}

\subsection{Closing comments and problems} In the generality it is stated, Problem~\ref{P:10} (i.e., \cite[Problem~10]{F4}) remains open except in the $\ell=1$ case. In this article, we first showed that the nilfactor of a system is characteristic for the corresponding sequence of iterates under the additional super niceness assumption. 
Second, we showed that the goodness property alone was enough to imply the required equidistribution properties. This comes as no surprise for, as we have already mentioned, the goodness property is a strong equidistribution notion. Hence, to completely resolve the problem, someone has to answer the following problem in the positive:

\begin{problem}\label{P:cp}
For $\ell\in \mathbb{N}$ let $(\mathcal{P}_N)_N=(p_{1,N},\ldots,p_{\ell,N})_N$ be a good sequence of $\ell$-tuples of polynomials. Is it true that for every system its nilfactor $\mathcal{Z}$ is characteristic for $(\mathcal{P}_N)_N$?
\end{problem}

Analogously, to solve Problem~\ref{P:2}, it suffices to answer the following:

\begin{problem}\label{P:P2}
Let $(p_{N})_N$ be a good sequence of polynomials. Is it true that for every $\ell\in \mathbb{N}$ and every system its nilfactor $\mathcal{Z}$ is characteristic for the sequence $(\mathcal{P}_N)_N=(p_{N},2p_{N}\ldots,\ell p_{N})_N$?
\end{problem}




As in our results we have convergence to the ``expected'' limit, it is reasonable 
for someone to study the corresponding pointwise results along natural numbers. So, we naturally close this article with the following problem:

\begin{problem}\label{P:last}
Find classes of good variable polynomial iterates (e.g., the ones in Theorems~\ref{T:main} and ~\ref{T:mc}) for which we have the corresponding pointwise convergence results.\footnote{ Someone can start by studying the pointwise convergence for the special cases of averages with iterates from Examples~\ref{Ex:1} and ~\ref{Ex:2}.}
\end{problem}


\section*{Acknowledgments} Thanks go to D.~Karageorgos with whom I started discussing the problem; N.~Frantzikinakis for his constant support and fruitful discussions during the writing of this article; and N.~Kotsonis for his detailed corrections on the
text. I am also deeply thankful to the anonymous Referees, X and Y, whose detailed feedback led to numerous clarifications, improving the readability and quality of the article.









\medskip
Received November 2021; revised March 2022; early access May 2022.
\medskip


\begin{thebibliography}{99}

\bibitem{B0} (MR891243) [10.1090/conm/065/891243]
\newblock V.~Bergelson,
\newblock \doititle{Ergodic Ramsey theory},
\newblock \emph{Logic and Combinatorics (Arcata, Calif., 1985)}, Contemp. Math. Amer. Math. Soc., Providence, RI, \textbf{65} (1987), 63--87.

\bibitem{Be} (MR912373) [10.1017/S0143385700004090]
\newblock V.~Bergelson,
\newblock \doititle{Weakly mixing PET},
\newblock \emph{Ergodic Theory Dynam. Systems}, \textbf{7} (1987), 337--349.

\bibitem{BK} (MR2545011) [10.1017/S0143385708000862]
\newblock V.~Bergelson and I.~H\r{a}land-Knutson,
\newblock \doititle{Weakly mixing implies mixing of higher orders along tempered functions},
\newblock \emph{Ergodic Theory Dynam. Systems}, \textbf{29} (2009), 1375--1416.

\bibitem{BL2} (MR1881925) [10.1007/s002220100179]
\newblock V.~Bergelson and A.~Leibman,
\newblock \doititle{A nilpotent Roth theorem},
\newblock \emph{Invent. Math.}, \textbf{147} (2002), 429--470.

\bibitem{BL} (MR1325795) [10.1090/S0894-0347-96-00194-4]
\newblock V.~Bergelson and A.~Leibman,
\newblock \doititle{Polynomial extensions of van der Waerden's and Szemer\'edi's theorems},
\newblock \emph{J. Amer. Math. Soc.}, \textbf{9} (1996), 725--753.

\bibitem{CFH} (MR2795725) [10.1112/plms/pdq037]
\newblock Q.~Chu, N.~Frantzikinakis and B.~Host,
\newblock \doititle{Ergodic averages of commuting transformations with distinct degree polynomial iterates},
\newblock \emph{Proc. Lond. Math. Soc.}, \textbf{102} (2011), 801--842.

\bibitem{DFMKS}
\newblock S.~Donoso, A.~Ferr\'e Moragues, A.~Koutsogiannis and W.~Sun,
\newblock {Decomposition of multicorrelation sequences and joint ergodicity},
\newblock preprint, 2021, \arXiv{2106.01058}.

\bibitem{DKS} (MR4092858) [10.1017/etds.2018.118]
\newblock S.~Donoso, A.~Koutsogiannis and W.~Sun,
\newblock \doititle{Pointwise multiple averages for sublinear functions},
\newblock \emph{Ergodic Theory Dynam. Systems}, \textbf{40} (2020), 1594--1618.

\bibitem{DKS2} [10.1007/s11854-021-0186-z]
\newblock S.~Donoso, A.~Koutsogiannis and W.~Sun,
\newblock \doititle{Seminorms for multiple averages along polynomials and applications to joint ergodicity},
\newblock \emph{J. d'Analyse Math.}, (2021)

\bibitem{F3} (MR3347186) [10.1090/S0002-9947-2014-06275-2]
\newblock N.~Frantzikinakis,
\newblock \doititle{A multidimensional Szemer\'{e}di theorem for Hardy sequences of different growth},
\newblock \emph{Trans. Amer. Math. Soc.}, \textbf{367} (2015), 5653--5692.

\bibitem{F1} (MR2585398) [10.1007/s11854-009-0035-y]
\newblock N.~Frantzikinakis,
\newblock \doititle{Equidistribution of sparse sequences on nilmanifolds},
\newblock \emph{J. Anal. Math.}, \textbf{109} (2009), 353--395.

\bibitem{F2} (MR2762998) [10.1007/s11854-010-0026-z]
\newblock N.~Frantzikinakis,
\newblock \doititle{Multiple recurrence and convergence for Hardy sequences of polynomial growth},
\newblock \emph{J. Anal. Math.}, \textbf{112} (2010), 79--135.

\bibitem{F4} (MR3613710)
\newblock N.~Frantzikinakis,
\newblock {Some open problems on multiple ergodic averages},
\newblock \emph{Bull. Hellenic Math. Soc.}, \textbf{60} (2016), 41--90.

\bibitem{F5} (MR3829173) [10.1093/imrn/rnx002]
\newblock N.~Frantzikinakis,
\newblock \doititle{An averaged Chowla and Elliott conjecture along independent polynomials},
\newblock \emph{Int. Math. Res. Not. IMRN}, \textbf{2018} (2018), 3721--3743.

\bibitem{FHK} (MR3047073) [10.1007/s11856-012-0132-y]
\newblock N.~Frantzikinakis, B.~Host and B.~Kra,
\newblock \doititle{The polynomial multidimensional Szemer\'edi theorem along shifted primes},
\newblock \emph{Israel J. Math.}, \textbf{194} (2013), 331--348.

\bibitem{Fu} (MR498471) [10.1007/BF02813304]
\newblock H.~Furstenberg,
\newblock \doititle{Ergodic behavior of diagonal measures and a theorem of Szemer\'{e}di on arithmetic progressions},
\newblock \emph{J. Analyse Math.}, \textbf{31} (1977), 204--256.

\bibitem{GT} (MR2877065) [10.4007/annals.2012.175.2.2]
\newblock B.~Green and T.~Tao,
\newblock \doititle{The quantitative behaviour of polynomial orbits on nilmanifolds},
\newblock \emph{Ann. of Math.}, \textbf{175} (2012), 465--540.

\bibitem{Hard}[10.1112/plms/s2-10.1.54]
\newblock G.~H.~Hardy,
\newblock \doititle{Proc. of the London Math. Society},
\newblock \emph{Proceedings of the London Mathematical Society}, \textbf{s2-10} (1912), 54--90.

\bibitem{HK99} (MR2150389) [10.4007/annals.2005.161.397]
\newblock B.~Host and B.~Kra,
\newblock \doititle{Nonconventional ergodic averages and nilmanifolds},
\newblock \emph{Annals of Math.}, \textbf{161} (2005), 397--488.

\bibitem{KK} (MR3999460) [10.4064/sm171102-18-9]
\newblock D.~Karageorgos and A.~Koutsogiannis,
\newblock \doititle{Integer part independent polynomial averages and applications along primes},
\newblock \emph{Studia Math.}, {\bf 249} (2019), 233--257.

\bibitem{K} (MR3809055) [10.3934/dcds.2018113]
\newblock Y.~Kifer,
\newblock \doititle{Ergodic theorems for nonconventional arrays and an extension of the Szemer\'edi theorem},
\newblock \emph{Discrete Contin. Dyn. Syst.}, \textbf{38} (2018), 2687--2716.

\bibitem{K0} (MR3774837) [10.1017/etds.2016.40]
\newblock A.~Koutsogiannis,
\newblock \doititle{Closest integer polynomial multiple recurrence along shifted primes},
\newblock \emph{Ergodic Theory Dynam. Systems}, {\bf 38} (2018), 666--685.

\bibitem{K1} (MR3789175) [10.1017/etds.2016.67]
\newblock A.~Koutsogiannis,
\newblock \doititle{Integer part polynomial correlation sequences},
\newblock \emph{Ergodic Theory Dynam. Systems}, \textbf{38} (2018), 1525--1542.

\bibitem{K2} (MR4201837) [10.3934/dcds.2020314]
\newblock A~Koutsogiannis,
\newblock \doititle{Multiple ergodic averages for tempered functions},
\newblock \emph{Discrete Contin. Dyn. Syst.}, \textbf{41} (2021), 1177--1205.

\bibitem{L} (MR2122919) [10.1017/S0143385704000215]
\newblock A.~Leibman,
\newblock \doititle{Pointwise Convergence of ergodic averages for polynomial sequences of translations on a nilmanifold},
\newblock \emph{Ergodic Theory Dynam. Systems}, \textbf{25} (2005), 201--213.

\bibitem{Rat} (MR1106945) [10.1215/S0012-7094-91-06311-8]
\newblock M.~Ratner,
\newblock \doititle{Raghunatan's topological conjecture and distribution of unipotent flows},
\newblock \emph{Duke Math. J.}, \textbf{63} (1991), 235--280.

\bibitem{W12} (MR2912715) [10.4007/annals.2012.175.3.15]
\newblock M.~Walsh,
\newblock \doititle{Norm convergence of nilpotent ergodic averages},
\newblock \emph{Annals of Math.}, \textbf{175} (2012), 1667--1688.

\bibitem{Ziegler} (MR2257397) [10.1090/S0894-0347-06-00532-7]
\newblock T.~Ziegler,
\newblock \doititle{Universal characteristic factors and Furstenberg averages},
\newblock \emph{J. Amer. Math. Soc.}, \textbf{20} (2007), 53--97.



\end{thebibliography}
\end{document}